\documentclass[10pt]{amsart}
\usepackage{amsmath,amsfonts,amsthm,amssymb,amsxtra,fancybox,graphicx,color,caption,subcaption}
\usepackage[pdfborder={0 0 0}]{hyperref}

\newcommand{\MSC}[1]{{\small\emph{AMS Subject Classification (2010): }#1}}
\renewcommand{\keywords}[1]{{\small\emph{Keywords: }#1}}
\newcommand{\Email}[1]{{\sl E-mail address:\/} {\rm\textsf{#1}}}
\newcommand{\be}[1]{\begin{equation}\label{#1}}
\newcommand{\ee}{\end{equation}}
\renewcommand{\(}{\left(}
\renewcommand{\)}{\right)}

\newcommand{\R}{\mathbb R}
\newcommand{\N}{\mathbb N}

\newcommand{\dd}{\;\mathrm{d}}
\newcommand{\barD}{{\phi_0}}
\newcommand{\barDM}{{\phi_0^M}}
\renewcommand{\L}{\mathcal L}
\newcommand{\E}{\mathcal E_\barD}
\newcommand{\F}{\mathcal F_M}
\newcommand{\eps}{\varepsilon}
\newcommand{\LinearOp}{\mathrm H_D}
\newcommand{\scalar}[2]{\left\langle#1,#2\right\rangle_{\!D}}
\newcommand{\LinearizedEnergyOp}{\mathrm E_\phi}
\newcommand{\infimum}[2]{\inf_{\begin{array}{c}#1\cr #2\end{array}}}
\newcommand{\seq}[2]{({#1}_{#2})_{#2\in\N}}

\newtheorem{thm}{Theorem}
\newtheorem{lem}[thm]{Lemma}
\newtheorem{prop}[thm]{Proposition}
\newtheorem{cor}[thm]{Corollary}
\newtheorem{rem}[thm]{Remark}

\newcommand{\DeuxFigures}[5]{\hspace{-0.45\textwidth}\begin{figure}[#1]\begin{center}\makebox[\textwidth][c]{\hfill\begin{subfigure}{0.4\textwidth}\includegraphics[width=\textwidth]{#2}\end{subfigure}\hfill\begin{subfigure}{0.4\textwidth}\includegraphics[width=\textwidth]{#3}\end{subfigure}\hfill}\vspace*{-4pt}\caption{\small #4}\label{#5}\end{center}\end{figure}}
\newcommand{\TroisFigures}[6]{
\begin{figure}[#1]\begin{center}\makebox[\textwidth][c]{
\hfill \begin{subfigure}{0.35\textwidth}\includegraphics[width=\textwidth]{#2}\end{subfigure}
\hfill \begin{subfigure}{0.35\textwidth}\includegraphics[width=\textwidth]{#3}\end{subfigure}
\hfill\begin{subfigure}{0.35\textwidth}\includegraphics[width=\textwidth]{#4}\end{subfigure}\hfill}\vspace*{-4pt}
\caption{\small #5}\label{#6}\end{center}\end{figure}}

\begin{document}
\title[Two crowd motion and herding models]{Stationary solutions of Keller-Segel type crowd motion and herding models: multiplicity and dynamical stability}

\author[J.~Dolbeault]{Jean Dolbeault}
\address[J.~Dolbeault]{Ceremade (UMR CNRS no. 7534), Universit\'e Paris-Dauphine, Place de Lattre de Tassigny, F-75775 Paris C\'edex 16, France. \Email{dolbeaul@ceremade.dauphine.fr}}

\author[G.~Jankowiak]{Gaspard Jankowiak}
\address[G.~Jankowiak]{Ceremade (UMR CNRS no. 7534), Universit\'e Paris-Dauphine, Place de Lattre de Tassigny, F-75775 Paris C\'edex 16, France. \Email{jankowiak@ceremade.dauphine.fr}}

\author[P.~Markowich]{Peter Markowich}
\address[P.~Markowich]{DAMTP, University of Cambridge Cambridge CB3 0WA, UK.\newline \Email{P.A.Markowich@damtp.cam.ac.uk}}

\date{\today}

\begin{abstract} In this paper we study two models for crowd motion and herding. Each of the models is of Keller-Segel type and involves two parabolic equations, one for the evolution of the density and one for the evolution of a mean field potential. We classify all radial stationary solutions, prove multiplicity results and establish some qualitative properties of these solutions, which are characterized as critical points of an energy functional. A notion of variational stability is associated to such solutions.
\par The dynamical stability in a neighborhood of a stationary solution is also studied in terms of the spectral properties of the linearized evolution operator. For one of the two models, we exhibit a Lyapunov functional which allows to make the link between the two notions of stability. Even in that case, for certain values of the mass parameter and all other parameters taken in an appropriate range, we find that two dynamically stable stationary solutions exist. We further discuss qualitative properties of the solutions using theoretical methods and numerical computations.\\[4pt]
\\[4pt]
\MSC{35K40; 35J20; 35Q91}
\end{abstract}

\maketitle\thispagestyle{empty}

\section{Introduction}\label{Sec:Intro}

The Keller-Segel model in chemotaxis has attracted lots of attention over the last years. However, most of the theoretical results have been obtained either in a parabolic-elliptic setting or when the coefficients, like the chemosensitivity coefficient, are independent of the solution. Models used in biology usually involve coefficients which depend on the solution itself, thus making the problems far more nonlinear, and also far less understood. The crowd motion and herding models considered here are two problems in the same class, where the main additional features compared to the standard version of the Keller-Segel model are the limitation (prevention of overcrowding) of the drift for the mass density in both models, and the limitation of the source in the equation for the chemo-attractant in one of the two models. Such limitations have important consequences: there is multiplicity of solutions for a given mass, in certain regimes; plateau-like solutions have an interesting pattern for modeling issues; the flux limitation forbids concentration and guarantees nice properties of the solutions, but also raises non-trivial stability issues concerning the set of stationary solutions, which we investigate numerically. The two models can be considered as test cases for the understanding of a very large class of parabolic-parabolic systems with the property of having several attractors. The fact that radial solutions are bounded and can be fully parametrized in relatively simple terms makes the study tractable. Most of the difficulties come from the complicated dependence of the solutions on the total mass, which is the crucial parameter in the two cases. Numerically, the difficulty comes from the parameters of the model which have to be chosen in ranges that make the problem rather stiff.

\subsection{Description of the models}\label{Sec:Description}

In this paper, we shall consider herding and crowd motion models describing the evolution of a density $\rho$ of individuals subject to a drift $\nabla D$ and confined to a bounded, open set $\Omega\subset\R^d$. The evolution equation for $\rho$ is given by
\be{rho}
\partial_t\rho=\Delta\rho-\nabla\cdot\big(\rho\,(1-\rho)\,\nabla D\big)
\ee
where $\rho_t$ stands for the derivative of $\rho$ with respect to time $t$ and the $\rho\,(1-\rho)$ includes the prevention of overcrowding term. For an isolated system, it makes sense to introduce a no-flux boundary condition, that is
\be{bcrho}
\big(\nabla\rho-\rho\,(1-\rho)\,\nabla D\big)\cdot\nu=0\quad\mbox{on}\quad\partial\,\Omega
\ee
which guarantees the conservation of the number of individuals (or conservation of mass), namely that
\be{mass_constraint}
\int_\Omega \rho\;\dd x = M
\ee
is independent of $t$. In the models considered in this paper, we shall assume that the potential $D$ solves a parabolic equation
\be{D}
\partial_tD=\kappa\,\Delta D-\delta\,D+g(\rho)
\ee
and is subject to homogeneous Neumann boundary conditions
\be{bcD}
\nabla D\cdot\nu=0\quad\mbox{on}\quad\partial\,\Omega\;.
\ee
We restrict our purpose either to \emph{Model (I)} when
\be{ModelI}
g(\rho)=\rho\,(1-\rho)
\ee
or to \emph{Model (II)} when
\be{ModelII}
g(\rho)=\rho\;.
\ee
In this paper, our purpose is to characterize stationary solutions and determine their qualitative properties.

\subsection{Motivations}\label{Sec:Motivations}

Human crowd motion models are motivated by the will to prevent stampedes in public places mainly by implementing a better design of walkways. Most crowd motion models do not convey herding effects well enough, that is, loosely speaking, when people bunch up and try to move in the same direction, as typically occurs in emergency situations.

In an effort to improve herding and crowd motion models, M.~Burger \emph{et al.}~in~\cite{bmp2011} have derived Model~(I) and Model (II) as the continuous limit of a microscopic cellular automaton model introduced in~\cite{Kirchner2002}. It takes the form a parabolic-parabolic system for the density of people $\rho$ and for the field $D$, where $D$ is a mean field potential which carries the herding effects. Basically, people are subject to random motion, with a preference for moving in the direction others are following. Random effects are taken into account by a diffusion, while a drift is created by the potential $D$, which accounts for locations that are or were previously occupied. To account for the packing of the people, empty spaces are preferred, which explains the role of the $(1-\rho)$ term in front of the drift, with $1$ being the maximal density. Such a correction is refered to as \emph{prevention of overcrowding} in the mathematical literature.

Both quantities $\rho$ and $D$ undergo diffusion which happens much faster for $\rho$, this point being reflected by the fact that the constant $\kappa$ is assumed to be small. The potential $D$ decays in time with rate $\delta>0$ and increases proportionally to the density $\rho$, but only if the density is not too high in case of Model (I): this is taken into account by the source term $g(\rho)$ given either by \eqref{ModelI} or by \eqref{ModelII}. As we shall see in this paper, interesting phenomena occur when $\delta$ is also taken small.

In many aspects, these models are quite similar to the Keller-Segel model used in chemotaxis. \emph{Prevention of overcrowding} has already been considered in several papers, either in the parabolic-elliptic case in \cite{MR2745794,MR2397995} or in the parabolic-parabolic case in \cite{di2008fully} (with a diffusion dominated large time asymptotics) and \cite{MR2745794} (where, additionally, the case of several species and cross-diffusion was taken in to account). In these papers the emphasis was put on the asymptotic behaviors, with a discussion of the possible asymptotic states and behaviors depending on the nonlinearities in \cite{MR2274484} and a study of plateau-like quasi-stationary solutions and their motion in \cite{MR2397995}. This of course makes sense when the domain is the entire space, but a classification of the stationary solutions in bounded domains and in particular plateau-like solutions is still needed, as it is strongly suggested by \cite{bmp2011} that such solutions have interesting properties, for instance in terms of stability.

Because of the $(1-\rho)$ factor in front of the drift, the transport term vanishes in our models as $\rho$ approaches 1, so that for any initial data bounded by~$1$, the density remains bounded by~$1$. Hence blow-up, which is a major difficulty for the analysis of the usual Keller-Segel system for masses over~$8\pi$ (\emph{cf.} for instance~\cite{blanchet2006two}), does not occur here. In contrast with the parabolic-elliptic Keller-Segel model with prevention of overcrowding studied in \cite{MR2745794}, Models (I) and~(II) are based on a system of coupled parabolic equations. This has interesting consequences for the evolution problem as, \emph{e.g.}, it introduces memory effects. It also has various consequences for the dynamical stability of the stationary states. In Model~(I), the source term in the equation for $D$ involves $\rho\,(1-\rho)$ instead of~$\rho$. Such a nonlinear source term introduces additional difficulties as, for instance, no Lyapunov functional is known up to now.

\subsection{Main results}\label{Sec:Intro-Main}

Let us summarize some of the main results of this paper, in case of Models (I) and~(II), when $\Omega$ is a ball, as far as radial nonnegative stationary solutions are concerned. As we shall see below the stationary solutions of interest are  either constants or monotone functions, which are then plateau-like.
\begin{thm}\label{Thm:Main} Let $\Omega$ be a ball and consider solutions of Models (I) and~(II) subject to boundary conditions \eqref{bcrho} and \eqref{bcD}. Then the masses of the radial nonnegative stationary solutions as defined by \eqref{mass_constraint} range between $0$ and $|\Omega|$ and we have:
\begin{enumerate}
\item[(i)] Non constant stationary solutions exists only for $M$ in a strict sub-interval $(0,|\Omega|)$,
\item[(ii)] Constant solutions are variationally and dynamically unstable in a strictly smaller subinterval,
\item[(iii)] There is a range of masses in which only non-constant stationary solutions are stable, given by the condition that $\kappa\,\lambda_1 + \delta$ is small enough, where $\lambda_1$ denotes the lowest positive eigenvalue of $-\Delta$ in $\Omega$ subject to Neumann homogeneous boundary conditions,
\item[(iv)] For any given mass, variationally stable stationary solutions with low energy are either monotone or constant; in case of Model~(II), monotone, plateau-like solutions are then stable and attract all low energy solutions of the evolution problem in a certain a range of masses.
\end{enumerate}
\end{thm}
Much more can be said on stationary solutions, as we shall see below and some of our results are not restricted to radial solutions on a ball. The natural parameter for the solutions of Models (I) and~(II) is $M$, but it is much easier to parametrize the set of solutions by an associated Lagrange multiplier: see Section~\ref{Sec:Solutions}. In particular, stationary solutions are then critical points of an energy defined in Section~\ref{Sec:EnergyUnConstrained}, and there is a notion of \emph{variational stability} associated to this energy. Taking into account the mass constraint, as it is done in Section~\ref{Sec:EnergyConstrained}, makes the problem more difficult. To study the evolution problem, one can rely on a Lyapunov functional introduced in Section~\ref{Sec:Lyapunov}, but only in case of Model (II). \emph{Dynamical stability} is studied through the spectrum of the linearized evolution operator in Section~\ref{Sec:LinearizedEvolution} and the interplay between notions of variational and dynamical stability is also studied in details. How to harmonize the two points of view on stability is a question that Model (I) and Model (II) share with all parabolic-parabolic models of chemotaxis. In case of Model (II), results are summarized in Theorem~\ref{Thm:Main2}. The issue of the stability of monotone -- constant or non constant -- solutions is a subtle question and most of this paper is devoted to this point. Precise definitions of variational and dynamical stability will be given later on.

Numerical results go beyond what can be proved rigorously. Because we use the parametrization by the Lagrange multiplier, we are able to compute \emph{all} radial solutions. In practice, we shall focus on the role of constant and monotone plateau-like solutions. A list of detailed qualitative results is provided at the beginning of Section~\ref{Sec:Numerics}. Theoretical and numerical results are discussed in Section~\ref{Sec:Conclusion}.

\subsection{Some references}\label{Sec:References}

The two models considered in this paper have been introduced in \cite{bmp2011} at the PDE level. Considerations on the stability of constant solutions can be found there as well. Models (I) and~(II) involve a system of two parabolic equations, like the so-called parabolic-parabolic Keller-Segel system for which we primarily refer to \cite{MR2433703}. In such a model, stationary solutions have to be replaced by self-similar solutions, which also have multiplicity properties (see \cite{springerlink:10.1007/s00285-010-0357-5}). How the parabolic-parabolic model is related to the parabolic-elliptic case has been studied in \cite{MR2515582,MR2433703}. The parabolic-elliptic counterpart of Model (I) is known: for plateau solutions and the coarsening of the plateaus, we refer to~\cite{MR2397995} (also see \cite{MR2274484,MR2745794}; related models can be found in the literature under the name of Keller-Segel model with logistic sensitivity or congestion models).

One of the technical but crucial issues for a complete classification of all solutions is how to parametrize the set of solutions. Because Lyapunov or energy functionals are not convex, this is a by far more difficult issue than in the repulsive case, for which we refer to \cite{MR1847430}. The lack of convexity makes it difficult to justify but, at a formal level, the evolution equations in Model~(II) can be interpreted as gradient flows with respect to some metric involving a Wasserstein distance (see \cite{Carrillo-Lisini} in case of the Keller-Segel model and \cite{Blanchet-Laurencot} for a more general setting; also see~\cite{laurencot2011gradient} for an earlier result in the same spirit). To be precise, one has to consider the Wasserstein distance for $\rho$ and a $L^2$ distance for $D$ as in \cite{CC2012}. The difficulty comes from the fact that the Lyapunov functional is not displacement convex (see for instance \cite{blanchet2008convergence} and subsequent papers in the parabolic-elliptic case of the Keller-Segel system). Using methods introduced in \cite{MR2581977}, this may eventually be overcome but is still open at the moment, as far as we know.\newpage

\section{Radial stationary solutions}\label{Sec:Solutions}

\subsection{A parametrization of all radial stationary solutions}\label{Sec:parametrization}

Any stationary solution of \eqref{rho} solves
\[
\nabla\rho-\rho\,(1-\rho)\,\nabla D=0\quad\mbox{on}\quad\Omega\,,
\]
which means
\be{rhoD}
\rho=\frac1{1+e^{-\phi}}
\ee
where $\phi=D-\barD$ and $\barD\in\mathbb{R}$ is an integration constant determined by the mass constraint~\eqref{mass_constraint}: $\barD$ is the unique real number such that
\be{Eqn:LagrangeMultiplier}
\int_\Omega\frac1{1+e^{\barD-D}}\,\dd x=M\;.
\ee
Taking into account boundary conditions \eqref{bcD}, Eq.~\eqref{D} now amounts~to
\be{static}
-\,\kappa\,\Delta\phi+\delta\,(\phi+\barD)-f(\phi)=0\quad\mbox{on}\quad\Omega
\ee
with boundary conditions
\be{neumann}
\nabla \phi \cdot \nu=0\quad\mbox{on}\quad\partial\,\Omega\;.
\ee
The functions $f$ and $F$ are defined by $f=F'$ and
\[
F(\phi)=\rho=\frac1{1+e^{-\phi}}\quad\mbox{and}\quad f(\phi)=\rho\,(1-\rho)=\frac{e^{-\phi}}{(1+e^{-\phi})^2}\quad\mbox{in case of Model (I)}\;,
\]
\[
F(\phi)=\log(1+e^{\phi}) \quad\mbox{and}\quad f(\phi)=\rho=\frac1{1+e^{-\phi}}\quad\mbox{in case of Model (II)}\;.
\]
The crucial observation for our numerical computation is based on the following result.
\begin{prop}\label{Prop:Parametrization} If $\Omega$ is the unit ball in $\R^d$, $d\ge 2$, all radial solutions of \eqref{static}-\eqref{neumann} with $f$ as above are smooth and can be found by solving the shooting problem
\[
-\,\kappa\,\(\varphi_a''+\tfrac{d-1}r\,\varphi_a'\)+\delta\,(\varphi_a+\barD)-f(\varphi_a)=0\;,\quad\varphi_a'(0)=0\;,\quad\varphi_a(0)=a
\]
as a function of the parameter $a\in\R$. The shooting criterion is: $\varphi_a'(1)=0$.\par
If $d=1$, all solutions in $\Omega=(0,1)$ are given by the above ODE.
\end{prop}
\begin{proof} The proof presents no difficulty and is left to the reader.\end{proof}

\subsection{Constant solutions}\label{Sec:Constants}

Determining $\phi$ such that $\delta\,(\phi+\barD)-f(\phi)=0$, that is
\be{Eqn:CstSoln}
k(\phi):=\frac 1\delta\,f(\phi)-\phi=\barD\;,
\ee
exactly amounts to determine the (possibly multivalued) function $\barD\mapsto k^{-1}(\barD)$. The following result is not restricted to the special case of $f$ as defined in Model (I) or Model (II).
\begin{lem}\label{Lem:CstSoln} Let $\delta>0$. Assume that $f\in C^1(\R)$ is bounded and such that $\lim_{\phi\to\pm\infty} f'(\phi)=0$. Then the function $\phi\mapsto k'(\phi)=\frac 1\delta\,f'(\phi)-1$ has $2\,\ell$ zeros for some $\ell\in\N$ and the equation \eqref{Eqn:CstSoln} has at most $2\,\ell+1$ solutions. Moreover, for $|\phi_0|$ large enough, \eqref{Eqn:CstSoln} has one and only one solution, which is such that $\rho$ given by \eqref{rhoD} converges to $0$ as $\barD\to+\infty$ and to $1$ as $\barD\to-\infty$.\end{lem}
If $|\phi_0|$ is large, we observe that $k(\phi)\sim-\,\phi$. Other properties are elementary consequences of the intermediate values theorem and left to the reader. A plot is shown in Figure~\ref{Fig:CstSoln}.

With $f=F'$ and $f$ corresponding either to Model (I) or (II), all assumptions of Lemma~\ref{Lem:CstSoln} are satisfied with $\ell=0$ or $1$. For later purpose, let us define
\[
\phi_-(\barD):=\min k^{-1}(\barD)\quad\mbox{and}\quad\phi_+(\barD):=\max k^{-1}(\barD)
\]
and emphasize that $\phi_\pm$ depend on $\barD$. The set $k^{-1}(\barD)$ is reduced to a point if and only if $\phi_-(\barD)=\phi_+(\barD)$. From Lemma~\ref{Lem:CstSoln}, we also know that
\[
\phi_0^-:=\inf\{\barD\in\R\,:\,\phi_-(\barD)<\phi_+(\barD)\}\quad\mbox{and}\quad\phi_0^+:=\sup\{\barD\in\R\,:\,\phi_-(\barD)<\phi_+(\barD)\}
\]
are both finite.

Instead of parametrizing solutions by $\barD$, it is interesting to think in terms of mass. Here is a first result (see Fig.~\ref{Fig:Branches}) in this direction, which follows from the property that $k'(\phi_\pm(\barD))<0$ for any $\barD\in\R$.
\begin{lem}\label{Lem:Simple} Under the assumptions of Lemma~\ref{Lem:CstSoln}, $\barD\mapsto\phi_\pm(\barD)$ is monotone decreasing, and the corresponding masses are also monotone decreasing as a function of $\barD$.\end{lem}

The proof is elementary and left to the reader. If $\phi$ is a constant solution, it is a monotone increasing function of the mass according to \eqref{rhoD}. Hence the mass of a constant extremal solution $\phi=\phi_\pm(\barD)$ is a monotone decreasing function of $\barD$. Moreover, we have
\[
f'(\phi)=\rho\,(1-\rho)\,h(\rho)
\]
with $\rho$ given by \eqref{rhoD}, $h(\rho)=1-2\,\rho$ in case of Model (I) and $h(\rho)=1$ in case of Model (II). A simple computation shows that $\mathsf m:=\max_{\rho\in[0,1]}\rho\,(1-\rho)\,h(\rho)$ is equal to $1/6\,\sqrt 3$ and $1/4$ in case of Models (I) and~(II) respectively. As a consequence, with the notations of Lemma~\ref{Lem:CstSoln}, $\ell=0$ if either $\delta\ge\mathsf m$ or $\delta<\mathsf m$ and $\barD\in\R\setminus\big(\phi_0^-,\phi_0^+\big)$. If $\delta<\mathsf m$ we find that $\ell=1$ if $\barD\in\big(\phi_0^-,\phi_0^+\big)$: there are exactly 3 constant solutions.

In the case of Model (I) or (II), the (unique) constant solution taking values in $(\phi_-(\phi_0), \phi_+(\phi_0))$ is monotone increasing as a function of $\phi_0$ (when it exists), thus defining a range of masses in which Theorem~\ref{Thm:Main}~(iii) holds, as we shall see below.

\section{Unconstrained energy and constant solutions}\label{Sec:EnergyUnConstrained}

In this section we consider the problem for fixed $\barD$. On the space $\mathrm H^1(\Omega)$, let us define the \emph{energy functional} by
\be{energy}
\E[\phi]:=\tfrac{\kappa}2\int_{\Omega}|\nabla\phi|^2\,\dd x+\tfrac{\delta}2\int_{\Omega}|\phi+\barD|^2\,\dd x-\int_{\Omega}F(\phi)\;\dd x\;.
\ee
It is clear from \eqref{mass_constraint} that stationary solutions of Model (I) and Model (II) are critical points of $\E$ (see below Lemma~\ref{Lem:Variational}) for some given Lagrange multiplier $\barD$. Moreover, for a given $\barD$, we know how to compute all radial solutions as explained in Section~\ref{Sec:Solutions}. Hence we shall first fix $\barD$, study the symmetry of the minimizers of $\E$ and clarify the role of constant solutions.

\subsection{Critical points}\label{Subsec:CriticalPointsPhi0}

\begin{lem}\label{Lem:Variational} Assume that $F$ is Lipschitz continuous and $\Omega$ is bounded with $C^{1,\alpha}$ boundary for some $\alpha>0$. With $\barD$ kept constant, $\phi$ is a solution of \eqref{static}--\eqref{neumann} if and only if it is a critical point of $\E$ in $\mathrm H^1(\Omega)$.\end{lem}
It is straightforward to check that $\E$ has a minimizer for any given $\phi$, but such a minimizer is actually constant as we shall see below in Corollary~\ref{cor:bounds}. Non constant solutions are therefore not minimizers of $\E$, for fixed $\phi_0$. The regularity of the solution of \eqref{static}--\eqref{neumann} depends on the regularity of $F$, but when it is smooth as in the case of Models (I) and~(II), the standard elliptic theory applies and $\phi$ is smooth up to the boundary. We refer for instance to \cite{brezis2010functional} for a standard reference book. Details are left to the reader and we shall assume without further notice that solutions are smooth from now on.

Notice that our original problem is not set with $\barD$ fixed, but with mass constraint~\eqref{mass_constraint}. Understanding how results for a given $\barD$ can be recast into problems with~$M$ fixed is a major source of difficulties and will be studied in particular in Section~\ref{Sec:EnergyConstrained}.

\subsection{Linearized energy functional}

Consider the linearized energy functional
\[
\lim_{\eps\to0}\frac{\E[\phi+\eps\,\psi]-\E[\phi]}{2\,\eps^2}=\int_\Omega\psi\,(\LinearizedEnergyOp\,\psi)\;\dd x
\]
where $\phi$ is a stationary solution, $\psi\in\mathrm H^2(\Omega)$ and $\LinearizedEnergyOp\,\psi:=-\,\kappa\,\Delta\psi+\delta\,\psi-F''(\phi)\,\psi$. Notice that with $\rho$ given by~\eqref{rhoD}, we have
\be{Ephirho}
\LinearizedEnergyOp\,\psi=-\,\kappa\,\Delta\psi+\delta\,\psi-\rho\,(1-\rho)\,h(\rho)\,\psi
\ee
with $h(\rho)=1-2\,\rho$ in case of Model (I) and $h(\rho)=1$ in case of Model (II).

\subsection{Stability and instability of constant solutions}\label{Sec:StabVar}

Denote by $\seq\lambda n$ the sequence of all eigenvalues of $-\Delta$ with homogeneous Neumann boundary conditions, counted with multiplicity. The eigenspace corresponding to $\lambda_0=0$ is generated by the constants. Three constant solutions co-exist when constant solutions $\phi$ take their values in $k\circ(k')^{-1}(0,+\infty)$, that is when
\[
\delta-\rho\,(1-\rho)\,h(\rho)<0\;.
\]
A constant solution $(\rho,D=\phi+\barD)$ is \emph{variationally} unstable if $\LinearizedEnergyOp$ has a negative eigenvalue, that is if
\be{Ineq:ConstInstability}
\kappa\,\lambda_1+\delta-\rho\,(1-\rho)\,h(\rho)<0\;.
\ee
When such a condition is satisfied, the constant solution $\phi$ cannot be a local minimizer of $\E$. \emph{Dynamical} stability of the constant solutions with respect to the evolution governed by \eqref{rho}--\eqref{bcD} will be studied in Section~\ref{Sec:LinearizedEvolution}: in case of constant solutions, such an instability is also determined by \eqref{Ineq:ConstInstability}, as we shall see in Proposition~\ref{Prop:Equivalence}.

Condition~\eqref{Ineq:ConstInstability} is never satisfied if $\kappa\,\lambda_1+\delta\ge\textsf{m} := \max_{\rho\in[0,1]}\rho\,(1-\rho)\,h(\rho)$. Otherwise, this condition determines a strict subinterval of $(0,1)$ in terms of $\rho$, and hence an interval in $\phi$. This proves Theorem~\ref{Thm:Main}~(ii). A slightly more precise statement goes as follows.
\begin{lem}\label{Lem:CstSolnInstability} Let $\delta>0$. The set of values of $\barD$ for which there are constant solutions of~\eqref{static} which satisfy~\eqref{Ineq:ConstInstability} with $\rho$ given by~\eqref{rhoD} is contained in $\big(\phi_0^-,\phi_0^+\big)$. Moreover, if there exists a constant, variationally unstable solution, then there is also a constant, variationally stable solution of~\eqref{static} for the same value of~$\barD$, but with lower energy.\end{lem}
The proof of Lemma~\ref{Lem:CstSolnInstability} requires some additional observations. It will be completed in Section~\ref{Subsec:monotonicity}.

\subsection{Numerical range}\label{Sec:NumericalRange}
We shall postpone the proof of Lemma~\ref{Lem:CstSolnInstability} to the end of the next section. Cases of numerical interest studied in this paper are the following.
\begin{enumerate}
\item In dimension $d=1$ with $\Omega=(0,1)$, the first unstable mode is generated by $x\mapsto\cos(\pi\,x)$ and corresponds to $\lambda_1=\pi^2\approx9.87$.
\item In dimension $d=2$, the first positive critical point of the first Bessel function of the first kind $J_0$, that is $r_0:=\min\{r>0\,:\,J_0'(r)=0\}$, is such that $r_0\approx3.83$ so that $\lambda_{0,1}=r_0^2\approx14.68$ is an eigenvalue associated to the eigenspace generated by $r\mapsto J_0(r\,r_0)$. Applied to \eqref{Ineq:ConstInstability}, this determines the range of \emph{radial variational instability.} Recall that $J_0$ is the solution of $J_0''+\frac 1r\,J_0'+J_0=0$.

We may notice that non-radial instability actually occurs in a larger range, since the first positive critical point of the second Bessel function of the first kind $J_1$, that is $r_1:=\min\{r>0\,:\,J_1'(r)=0\}$, is such that $r_1\approx1.84$ so that $\lambda_{1,0}=r_1^2\approx3.39$ is an eigenvalue associated to the eigenspace generated by $r\mapsto J_1(r\,r_1)$, and $\lambda_1=\lambda_{1, 0} < \lambda_{0, 1}$. Applied to \eqref{Ineq:ConstInstability}, this determines the range of \emph{variational instability.} Recall that $J_1$ is the solution of $J_1''+\frac 1r\,J_1'-\frac 1{r^2}\,J_1'+J_1=0$.
\end{enumerate}
Let us notice that the values of $\max_{\rho\in[0,1]}\rho\,(1-\rho)\,h(\rho)$ are in practice also rather small, namely $1/6\,\sqrt 3\approx0.096$ and $1/4=0.25$ in case of Models (I) and~(II) respectively, which in practice, in view of the values of $\lambda_1$, makes the numerical computations rather stiff. In this paper we are interested in the qualitative behavior of the solutions and the role of the dimension, but not so much in the role of the surrounding geometry and hence we shall restrict our study to radial solutions. One of the advantages of dealing only with radial solutions is that we can use accurate numerical packages for solving ODEs and rely on shooting methods, thus getting a precise description of the solution set. Taking into account the effects of the geometry is another challenge but is, in our opinion, secondary compared to establishing all qualitative properties that can be inferred from our numerical computations. Another reason for restricting our study to radially symmetric functions is Proposition~\ref{Prop:Parametrization}: using the shooting method, we have the guarantee to describe all solutions, with additional informations like the knowledge of the range in which to adjust the shooting parameter, as a consequence of the observations of Section~\ref{Sec:Constants} (see also Proposition~\ref{prop:ranges}). Within the framework of radial solutions, we can henceforth give a thorough description of the set of solutions, that is clearly out of reach in more general geometries. However, as far as we deal with theoretical results, we will not assume any special symmetry of the solutions unless necessary.

In practice, numerical computations of this paper are done with $\delta=10^{-3}$ and $\kappa$ ranging from $5\times10^{-4}$ to $10^{-2}$. Such small values are dictated by \eqref{Ineq:ConstInstability}. They are also compatible with the computations and modeling considerations that can be found in \cite{bmp2011}. See Fig.~\ref{Fig:Branches} for a plot corresponding to a rather generic diagram representing constant solutions for Model (I) in dimension $d=1$. Numerically, our interest lies in the non-constant radial solutions that bifurcate from the constant solutions $\phi$ at threshold values for condition~\eqref{Ineq:ConstInstability}, that is for values of $\barD$ such that $\kappa\,\lambda_1+\delta-\rho\,(1-\rho)\,h(\rho)=0$ with $\lambda_1=\pi^2$ in dimension $d=1$, and $\lambda_1=\lambda_{0,1}$ when $d=2$. We shall take $\barD$ as the bifurcation parameter and compute the mass of the solution only afterwards, thus arriving at a simple parametrization of all solutions. Our main results are therefore a complete description of branches of solutions bifurcating from constant ones and giving rise to \emph{plateau} solutions. See Fig.~\ref{phi_plots} for some plots of the solutions. We notice that in the range considered for the parameters, the transition from high to low values is not too sharp. The numerical study will be confined to radial monotone solutions, but we will briefly explain in Section~\ref{Sec:OneDMonotone} (at least when $d=1$) what can be expected for solution with several plateaus. Concerning stability issues, decomposition on appropriate basis sets will be required, as will be explained in Section~\ref{Sec:Numerics}.

\subsection{Qualitative properties of the stationary solutions}

\begin{lem}\label{res:bounds} Let $\Omega$ be a bounded open set in $\R^d$ with $C^2$ boundary and assume that $k:\R\rightarrow\R$ is Lipschitz continous with
\[
\liminf_{u\to-\infty}k(u)>0\quad\mbox{and}\quad\limsup_{u\to+\infty}k(u)<0\;.
\]
Assume that all zeros of $k$ are isolated and denote them by $u_1<u_2<\ldots<u_N$ for some $N\ge1$. Then any solution of class~$C^2$ of $\Delta u+k(u)=0$ in $\Omega$, $\nabla u \cdot x = 0$ on $\partial\Omega$ takes values in $[u_1,u_N]$. \end{lem}
\begin{proof} Let $x^*\in\overline\Omega$ be a maximum point of $u$. We know that $-\Delta u(x^*)=k(u(x^*))\ge0$, even if $x^*\in\partial\,\Omega$ because of the boundary conditions. By assumption, we find that $u(x)\le u(x^*)\le u_N$ for any $x\in\overline\Omega$. Similarly, one can prove that $u\ge u_1$.\end{proof}

Applying Lemma~\ref{res:bounds} to \eqref{static}, \eqref{neumann} has straightforward but interesting consequences.
\begin{cor}\label{cor:bounds} Under the assumptions of Lemma~\ref{Lem:CstSoln}, for any given $\barD\in\R$, if $\phi$ is a solution of \eqref{static}--\eqref{neumann}, then we have that
\[
\phi_-(\phi_0)\le\phi(x)\le\phi_+(\phi_0)\quad\forall\,x\in\Omega\;.
\]
The minimum of $\E$ is achieved by a constant function. Moreover, if \eqref{Eqn:CstSoln} has only one solution $\phi$, then \eqref{static}--\eqref{neumann} also has only one solution, which is constant and $\phi\equiv\phi_-=\phi_+$. \end{cor}
\begin{proof}We simply observe that, according to the definition~\eqref{energy}, we have
\[
\E[\phi]\ge\tfrac{\delta}2\int_{\Omega}|\phi+\barD|^2\,\dd x-\int_{\Omega}F(\phi)\;\dd x
\]
and critical points of $\phi\mapsto \frac{\delta}{2}|\phi+\barD|^2-F(\phi)$ are precisely the constant solutions of \eqref{Eqn:CstSoln} with $f=F'$.\end{proof}

In the cases which are numerically studied in this paper, there is an additional property which is of particular interest.
\begin{prop}\label{prop:ranges} Consider either Model (I) or Model (II). Then there exists a constant unstable solution only if $\barD\in\big(\phi_0^-,\phi_0^+\big)$.\end{prop}
\begin{proof} This is an easy consequence of the properties of $f=F'$. Details are left to the reader.\end{proof}

\subsection{A monotonicity result}\label{Subsec:monotonicity}

For a given $\barD\in\R$, non-monotone radial functions always have higher energy $\E$ than radial monotone functions. We can state this observation as a slightly more general result, as follows.
\begin{prop}\label{prop:monotonicity} Assume that $\Omega$ is the unit ball in $\R^d$, $d\ge2$, and let $G\in W^{1,\infty}(\Omega)$. Then the functional $\mathcal G[\phi]:=\frac 12\int_\Omega|\nabla\phi|^2\,\dd x-\int_\Omega G(\phi)\,\dd x$ is bounded from below and for any radial non monotone function $\phi\in C^2(\Omega)$ satisfying \eqref{neumann}, with a finite number of critical points, there exists a radial monotone function~$\tilde\phi$ which satisfies \eqref{neumann}, coincides with~$\phi$ on a neighborhood of $0$ such that $\mathcal G[\tilde\phi]<\mathcal G[\phi]$.\end{prop}
\begin{proof} With a slight abuse of notations, we consider $\phi$ as a function of $r=|x|\in[0,1]$ and assume that it is solution of
\[
\phi''+G'(\phi)=-\frac{d-1}{r}\,\phi'\,.
\]
Multiplying by $\phi'$, we find that
\[
\frac{d}{dr}\(\frac 12\,{\phi'}^2+G(\phi)\)=-{\frac{d-1}{r}}\,{\phi'}^2<0\,.
\]
Unless $\phi$ is constant, assume that for some $r_0\in(0,1)$ we have $\phi'(r_0)=0$, and let $G_0:=G(\phi(r_0))$. Integrating on $(r_0,r)$, $r>r_0$, we find that
\[
\frac 12\,{\phi'}^2+G(\phi)<G_0\quad\mbox{on}\;(r_0,1)\;,
\]
and then $\frac 12\,{\phi'}^2-G(\phi)>{\phi'}^2-G_0>-\,G_0$ on $(r_0,R)$. Hence we have that
\[
\frac{\mathcal G[\phi]}{|\Omega|}=\int_0^{r_0}\(\frac12\,{\phi'}^2-G(\phi)\)\,r^{d-1}\,\dd r+\int_{r_0}^1\(\frac12\,{\phi'}^2-G(\phi)\)\,r^{d-1}\,\dd r>\frac{\mathcal G[\tilde\phi]}{|\Omega|}
\]
where $\tilde\phi$ defined by $\tilde\phi\equiv\phi$ on $(0,r_0)$ and $\tilde\phi\equiv\phi(r_0)$ on $(r_0,1)$. This concludes the proof.
\end{proof}

Proposition~\ref{prop:monotonicity} shows at ODE level why radial minimizers of the functional $\mathcal G$ have to be monotone. It is also preparatory for Lemma~\ref{lem:monotonicity:1d}.

\begin{proof}[Proof of Lemma~\ref{Lem:CstSolnInstability}] A constant solution which satisfies \eqref{Ineq:ConstInstability} cannot be a global minimizer for $\barD$ fixed. According to Corollary~\ref{cor:bounds}, there exists another constant solution under the assumptions of Lemma~\ref{Lem:CstSolnInstability}, which incidentally proves that $\phi_-(\barD)<\phi_+(\barD)$ with the notations of Section~\ref{Sec:Constants}. The fact that there is a constant stable solution with an energy lower than the energy of the unstable one is a consequence of Proposition~\ref{prop:monotonicity}.\end{proof}

Summarizing, for a given $\barD\in\R$, only constant solutions are to be considered for the minimization of~$\E$. However, the relevant problem in terms of modeling is the problem with mass constraint, at least in view of the evolution problem, and it is not as straightforward as the problem with a fixed Lagrange multiplier.

\section{Energy minimizers under mass constraint}\label{Sec:EnergyConstrained}

\subsection{Existence and qualitative properties of minimizers}

In this section, we assume that $M>0$ is fixed and consider $\barDM[D]=\barD$ uniquely determined by~\eqref{Eqn:LagrangeMultiplier}. Let us define the functional
\[
D\mapsto\F[D]:=\tfrac{\kappa}2\int_{\Omega}|\nabla D|^2\,\dd x+\tfrac{\delta}2\int_{\Omega}|D|^2\,\dd x-\int_{\Omega}F\(D-\barDM[D]\)\;\dd x\;.
\]
In such a case, $\barD$ can be seen as a Lagrange multiplier associated to the mass constraint and $\F[D]=\E[D-\barD]$.
\begin{prop}\label{Prop:Minimization} Assume that $F$ is a continuous function with a subcritical growth. If $\Omega$ is bounded with $C^{1,\alpha}$ boundary for some $\alpha>0$, then for any $M>0$, the functional $\F$ has at least one minimizer $D=\phi+\barD$ with $\barD=\barDM[D]$ in $H^1(\Omega)$, which is such that $\F[D]=\E[\phi]$, and $D$ is of class $C^\infty(\Omega)$ if~$F$ is of class $C^\infty$.\end{prop}
\begin{proof} It is straightforward to check that $\F$ has at least one minimizer in $H^1(\Omega)$ because any minimizing sequence converges up to the extraction of subsequences to a minimum $D=\phi+\barD$ by compactness and lower semi-continuity. Then $\phi$ is a critical point of $\E$ and regularity is a standard result of elliptic theory (see \emph{e.g.} \cite{brezis2010functional}) and bootstrapping methods. \end{proof}

In Model (I) and~(II), we respectively have $|F(\phi)|<1$ and $F(\phi)\in[0,\log(2)+\max(0,\phi)]$ so the assumptions of Proposition \ref{Prop:Minimization} are satisfied. Notice that it is not implied anymore that minimizers of $\F$ under mass constraint are constant functions and hence they might not be minimizers of $\E$.

\begin{lem}\label{res:extreme_mass} The mass of the density associated to non-constant solutions of \eqref{static}-\eqref{neumann} is bounded away from 0 and $|\Omega|$.\end{lem}
\begin{proof} Any non-constant solution of \eqref{static}-\eqref{neumann} has mass $M=\int_\Omega\frac1{1+e^{-\phi}}\,\dd x$ associated to its density, according to~\eqref{rhoD} and~\eqref{Eqn:LagrangeMultiplier}. Corollary~\ref{cor:bounds} gives the bounds
\[
M_-(\barD):=\frac{|\Omega|}{1+e^{-\phi_+(\barD)}}\le M\le\frac{|\Omega|}{1+e^{-\phi_-(\barD)}}=:M_+(\barD)\;.
\]
Let $M^{(-)}:=\min\left\{M_-(\barD)\,:\,\barD\in\big(\phi_0^-,\phi_0^+\big)\right\}$ and $M^{(+)}:=\max\left\{M_+(\barD)\,:\,\barD\in\big(\phi_0^-,\phi_0^+\big)\right\}$. Since $\barD\mapsto M_\pm(\barD)$ is a continuous function on $\R$, we know from Lemma~\ref{Lem:CstSolnInstability} that $\big(M^{(-)},M^{(+)}\big)$ is compactly included in $(0,|\Omega|)$. From Lemma~\ref{Lem:Simple}, we deduce that $M^{(\pm)}=M_\pm(\phi_0^\mp)$.\end{proof}

Notice that Lemma~\ref{res:extreme_mass} proves Theorem~\ref{Thm:Main}~(i).

\begin{cor}\label{res:range_mass} With above notations, we have $0<M^{(-)}\le M^{(+)}<1$ and minimizers of $\F$ are constant functions if $M\in(0,M^{(-)})\cup(M^{(+)},1)$. There is a subinterval of $(M^{(-)},M^{(+)})$ in which minimizers of~$\F$ are non constant functions.\end{cor}
Whether minimizers of $\F$ are constant solutions or not for some $M\in(M^{(-)},(M^{(+)})$ will be investigated numerically. For small masses, or masses close to the maximal mass $|\Omega|$ corresponding to the limit density $\rho=1$, we can state one more result.
\begin{cor}\label{Cor:CstSoln} Under the assumptions of Lemma~\ref{Lem:CstSoln}, with $M^{(\pm)}\in(0,|\Omega|)$ defined as above, there is one and only one solution $\phi$ of \eqref{rho}--\eqref{bcD} with mass $M\in(0,M^{(-)})\cup(M^{(+)},|\Omega|)$, and this solution is constant, given by $\phi=-\log\big(\frac{|\Omega|}M-1\big)$.\end{cor}

\subsection{A partial symmetry result}

\begin{lem}\label{lem:symmetry} Assume that $d=2$. If $\Omega$ is a disk, minimizers of $\F$ are symmetric under reflection with respect to a line which contains the origin. \end{lem}
\begin{proof} The proof of this lemma is inspired by \cite{lopes1996radial}. Assume that $\Omega$ is the unit disk centered at the origin and denote by $(x_1,x_2)$ cartesian coordinates in $\R^2$. Let us also define the open upper half-disk $\Omega_+:=\{x\in\Omega\,:\,x=(x_1,x_2)\,,\;x_1>0\}$. If $\phi$ is a minimizer of $\F$, we define $\tilde\phi$ by $\tilde\phi(x_1,x_2)= \phi(|x_1|,x_2)$, so that $\tilde\phi$ is symmetric with respect to the line $x_1=0$.
Up to a rotation, we can assume that $\Omega_+$ accounts for exactly half of the mass, \emph{i.e.} $\int_{\Omega_+}(1+e^{-\phi})^{-1}\dd x=M/2$,
so that $\int_\Omega(1+e^{-\tilde\phi})^{-1}\dd x=M$. Then, up to a reflection, we can assume that $\Omega_+$ accounts for at most half of the the value of $\F$:
\[
\frac{\kappa}2\int_{\Omega_+}|\nabla\phi|^2\,\dd x+\frac{\delta}2\int_{\Omega_+}|\phi+\barD|^2\,\dd x-\int_{\Omega_+}F(\phi)\;\dd x\le\frac 12\,\F[\phi]\;.
\]
It is then clear that $\tilde\phi$ is a minimizer of $\F$ such that the mass constraint \eqref{mass_constraint} is satisfied. As such, $\tilde\phi$ also solves the Euler-Lagrange equations, with same Lagrange multiplier $\barD$ because $\phi$ and $\tilde\phi$ coincide on $\Omega_+$. Then $w:=\phi-\tilde\phi$ solves the equation
\[
-\kappa\,\Delta w +h\,w=0\quad\mbox{with}\quad h:=\frac{\delta\,(\phi-\tilde\phi)+F'(\tilde\phi)-F'(\phi)}{\phi-\tilde\phi}
\]
on $\Omega$. Since $F\in C^\infty$ and $\phi$, $\tilde\phi$ are continuous, $h$ is bounded. According to \cite[Theorem 8.9.1]{hormander1969linear}, H\"ormander's uniqueness principle applies. Since $w\equiv0$ on $\Omega_+$, we actually have $w\equiv0$ on the entire disk~$\Omega$, and so $\phi=\tilde\phi$.\end{proof}

In higher dimensions, when $\Omega$ is a ball, the method can be extended and shows the symmetry of the solutions with respect to hyperplanes. Thus proving a result of so-called Schwarz foliated symmetry. The method also applies to the functional $\E$ with fixed $\barD$ and shows that a minimizer is radially symmetric, but this is useless as we already know that the minimum is achieved among constant solutions.

\subsection{One dimensional minimizers are monotone}\label{Sec:OneDMonotone}

One-dimensional stationary solution solve an autonomous ODE. This has several interesting consequences.
\begin{lem}\label{lem:monotonicity:1d} Let $d=1$ and $M>0$. Then minimizers of $\F$ are monotone, either increasing or decreasing.\end{lem}
\begin{proof} Assume that $\phi$ is a minimizer of $\F$ and $\Omega=(0,1)$. If $\phi$ is not monotone, it has a finite number of extremal points $0=r_0<r_1\ldots<r_N=1$ for some $N>1$. By uniqueness of the solution of the initial value problem with $\phi(r_i)$ given and $\phi'(r_i)=0$ we conclude that $\phi(r_i-s)=\phi(r_i+s)$ as long as $0\le r_i-s$ and $r_i+s\le1$, so that $r_i = \frac{i}{N}$, that is $\phi$ is $\frac{1}{N}$-periodic. With $\tilde\phi(r):=\phi(r/N)$, $r\in(0,1)$, we find that
\[
\int_0^1|\tilde\phi'|^2 \dd r=\frac1N\int_0^{1/N}|\phi'|^2 \dd r=\frac1{N^2}\int_0^1|\phi'|^2 \dd r<\int_0^1|\phi'|^2 \dd r\;,
\]
thus proving that $\E[\tilde\phi]<\E[\phi]$ while $\int_0^1(1+e^{-\tilde\phi}) \dd r=\int_0^1(1+e^{-\phi}) \dd r$, a contradiction.\end{proof}

{}From the scaling in the above proof, it is now clear that \emph{all} non-monotone one-dimensional solutions can be built from monotone ones by symmetrizing them with respect to their critical points, duplicating and scaling them. The intuitive idea is simple but giving detailed statements is unnecessarily complicated, so we will focus on monotone, or \emph{one-plateau} solutions.

\section{A Lyapunov functional}\label{Sec:Lyapunov}

In case of Model (II), let us consider the functional
\[
\L[\rho,D]:=\int_\Omega\left[\rho\,\log\rho+(1-\rho)\,\log(1-\rho)-\rho\,D\right]\,\dd x+\frac\kappa2\int_\Omega|\nabla D|^2\,\dd x+\frac\delta2\int_\Omega D^2\,\dd x\;.
\]
\begin{prop} The functional $\L$ is a \emph{Lyapunov functional} for Model (II) and if $(\rho,D)$ is a solution of \eqref{rho}--\eqref{bcD} and \eqref{ModelII}, then
\[
\frac d{dt}\,\L[\rho(t,\cdot),D(t,\cdot)]=-\int_\Omega\frac{|\,\nabla\rho-\rho\,(1-\rho)\,\nabla D\,|^2}{\rho\,(1-\rho)}\;\dd x-\int_\Omega|-\,\kappa\,\Delta D+\delta\,D-\rho\,|^2\,\dd x\le 0\;.
\]
As a consequence, any critical point of $\L$ under the mass constraint~\eqref{mass_constraint} is a stationary solution of \eqref{rho}--\eqref{bcD} and~\eqref{ModelII}, and any solution converges to a stationary solution. If $\Omega$ is a ball and if the initial datum is radial, then the limit is a radial stationary solution.\end{prop}
\begin{proof} An elementary computation shows that
\[
\frac d{dt}\,\L[\rho(t,\cdot),D(t,\cdot)]=-\int_\Omega\left[\log\Big(\frac\rho{1-\rho}\Big)-D\right]\,D_t\;\dd x-\int_\Omega(-\,\kappa\,\Delta D+\delta\,D-\rho)\,D_t\;\dd x
\]
and the expression of $\frac d{dt}\,\L[\rho(t,\cdot),D(t,\cdot)]$ follows from \eqref{rho}--\eqref{bcD}. Let $\rho_n(t,x):=\rho(t+n,x)$ and $D_n(t,x):=D(t+n,x)$. Since $\L$ is bounded from below, we have that
\[
\lim_{n\to\infty}\int_0^1\(\int_\Omega\frac{|\,\nabla\rho_n-\rho_n\,(1-\rho_n)\,\nabla D_n\,|^2}{\rho_n\,(1-\rho_n)}\;\dd x+\int_\Omega|-\,\kappa\,\Delta D_n+\delta\,D_n-\rho_n\,|^2\,\dd x\)dt=0\;,
\]
which proves that $(\rho_n,D_n)$ strongly converges to a stationary solution. Other details of the proof are left to the reader.\end{proof}

\begin{prop}\label{Prop:Equivalence} Let $M>0$ and consider Model (II). For any $D\in\mathrm H^1(\Omega)$, let $\barD$ be the unique real number determined by the mass constraint \eqref{Eqn:LagrangeMultiplier}. Then for any nonnegative $\rho\in\mathrm L^1(\Omega)$ satisfying the mass constraint~\eqref{mass_constraint}, we have
\[
\L[\rho,D]\ge\E[D-\barD]\;,
\]
and equality holds if and only if $\rho$ is given by \eqref{rhoD}, \emph{i.e.}~$\rho=1/(1+e^{-\phi})$, with $\phi=D-\barD$. As a consequence, for any minimizer $(\rho,D)$ of $\L$ satisfying~\eqref{mass_constraint}, $\rho$ is given by \eqref{rhoD} with $\phi=D-\barD$, $\barD$ satisfying the mass constraint~\eqref{Eqn:LagrangeMultiplier}, and $\phi$ is a minimizer of $\E[\phi]=\L[\rho,D]$. \end{prop}
\begin{proof} We only need to notice that the minimum of $\L[\rho,D]$ with respect to $\rho$ under the mass constraint~\eqref{mass_constraint} satisfies
\[
\log\Big(\frac\rho{1-\rho}\Big)=D-\barD\;.
\]
The completion of the proof follows from elementary computations which are left to the reader.
\end{proof}

\section{The linearized evolution operator}\label{Sec:LinearizedEvolution}

\subsection{Dynamical instability of constant solutions}

Assume that $(\rho,D)$ is a stationary solution of \mbox{\eqref{rho}--\eqref{bcD}}. Because of \eqref{rhoD} and \eqref{Eqn:LagrangeMultiplier}, the solution is fully determined by $D$. Let us consider a time dependent perturbed solution of the form $(\rho+\eps\,u,D+\eps\,v)$. Up to higher order terms, $u$ and $v$ are solutions of the linearized system
\be{LinearizedFlow}\left\{\begin{array}{l}
u_t=\nabla\cdot\big(\nabla u-(1-2\,\rho)\,u\,\nabla D-\,\rho\,(1-\rho)\,\nabla v\big)\,,\\[6pt]
v_t=\kappa\,\Delta v-\delta\,v+\,h(\rho)\,u
\end{array}\right.
\ee
with $h(\rho)=1-2\,\rho$ in case of Model (I) and $h(\rho)=1$ in case of Model (II). For later use, we introduce the notation $\LinearOp$ for the linear operator corresponding to the right hand side, so that we shall write
\[
(u,v)_t=\LinearOp\,(u,v)=\(\begin{array}{c}\LinearOp^{(1)}\,(u,v)\cr\LinearOp^{(2)}\,(u,v)\end{array}\)\quad\mbox{with}\quad\left\{\begin{array}{l}
\LinearOp^{(1)}\,(u,v)=\nabla\cdot\left[\rho\,(1-\rho)\,\nabla\big(\frac u{\rho\,(1-\rho)}-v\big)\right]\,,\\[6pt]
\LinearOp^{(2)}\,(u,v)=\kappa\,\Delta v-\delta\,v+\,h(\rho)\,u\;.\end{array}\right.
\]

Dynamical instability of constant solutions can be studied along the lines of \cite{bmp2011}. Let us state a slightly more general result. We are interested in finding the lowest possible $\mu$ in the eigenvalue problem
\be{Eqn:DynStabConstants}
-\,\LinearOp\,(u,v)=\mu\,(u,v)\;,
\ee
where $\LinearOp^{(1)}\,(u,v)$ now takes a simplified form, using the fact that $\rho$ is a constant:
\[
\LinearOp^{(1)}\,(u,v)=\Delta u-\rho\,(1-\rho)\,\Delta v\;.
\]
The condition $\mu<0$ provides a dynamically unstable mode. As in Section~\ref{Sec:StabVar}, let us denote by $\seq\lambda n$ the sequence of all eigenvalues of $-\Delta$ with homogeneous Neumann boundary conditions, counted with multiplicity, and by $\seq\phi n$ an associated sequence of eigenfunctions. If $u=\sum_{n\in\N}\alpha_n\,\phi_n$ and $v=\sum_{n\in\N}\beta_n\,\phi_n$, Problem~\eqref{Eqn:DynStabConstants} can be decomposed into
\begin{eqnarray*}
&&-\lambda_n\,\alpha_n+\rho\,(1-\rho)\,\lambda_n\,\beta_n=-\,\mu_n\,\alpha_n\\
&&-\kappa\,\lambda_n\,\beta_n-\delta\,\beta_n+h(\rho)\,\alpha_n=-\,\mu_n\,\beta_n
\end{eqnarray*}
for any $n\in\N$, that is
\begin{eqnarray*}
&&(\mu_n-\lambda_n)\,\alpha_n+\rho\,(1-\rho)\,\lambda_n\,\beta_n=0\\
&&h(\rho)\,\alpha_n+(\mu_n-\kappa\,\lambda_n-\delta)\,\beta_n=0
\end{eqnarray*}
which has non-trivial solutions $\alpha_n$ and $\beta_n$ if and only if the discriminant condition
\[
(\mu_n-\lambda_n)\,(\mu_n-\kappa\,\lambda_n-\delta)-\rho\,(1-\rho)\,h(\rho)\,\lambda_n=0
\]
is satisfied. This determines $\mu_n$ for any $n\in\N$, and the spectrum of $\LinearOp$ is then given by $\seq\mu n$. Collecting these observations, we can state the following result.
\begin{prop}\label{Prop:DynUnstabConstants} With the above notations, $\inf_{n\ge1}\mu_n<0$ if and only if \eqref{Ineq:ConstInstability} holds.\end{prop}
\begin{proof} The discriminant condition can be written as
\[
\mu_n^2-\left[(\kappa+1)\,\lambda_n+\delta\right]\,\mu_n+\lambda_n\,\left[\kappa\,\lambda_n+\delta-\rho\,(1-\rho)\,h(\rho)\right]=0
\]
so that there is a negative root if $\lambda_n\,(\kappa\,\lambda_n+\delta-\rho\,(1-\rho)\,h(\rho))<0$. Since $\seq\lambda n$ is nondecreasing and $\lambda_0=0$, there is at least one negative eigenvalue for \eqref{Eqn:DynStabConstants} if the above condition is satisfied with $n=1$. This concludes the proof.\end{proof}
In other words, dynamical instability of the constant solutions implies their variational instability. As we shall see numerically, variational and dynamical instability are not anymore equivalent for plateau-like solutions.

Notice that $\lambda_0=0$ has anyway to be excluded, as it corresponds to the direction generated by constants. Because of \eqref{Ineq:ConstInstability} we can ensure that the perturbation has zero average. This will be further discussed below, in the general case of a stationary solution.

\subsection{Variational criterion}

In the case of Model (II), we can look at the Lyapunov functional $\mathcal L$ and linearize it around a stationary solution $(\rho,D)$. Let
\[
\mathrm L_D[u,v]:=\lim_{\eps\to0}\frac{\mathcal L[\rho+\eps\,u,D+\eps\,v]-\mathcal L[\rho,D]}{2\,\eps^2}\;.
\]
A simple computation shows that
\[
\mathrm L_D[u,v]=\int_\Omega\(\frac{u^2}{2\,\rho\,(1-\rho)}-u\,v\)\dd x+\frac\kappa2\int_\Omega|\nabla v|^2\,\dd x+\frac\delta2\int_\Omega v^2\,\dd x\;.
\]
With $\LinearizedEnergyOp$ defined by~\eqref{Ephirho}, let
\be{Eqn:lambda}
\Lambda:=\infimum{\int_\Omega v\,\rho\,(1-\rho)\,\dd x=0}{v\not\equiv0}\frac{\int_\Omega v\,(\LinearizedEnergyOp\,v)\,\dd x}{\int_\Omega v^2\,\dd x}
\ee
with $\phi=D-\barD$, and $\barD$ satisfying \eqref{Eqn:LagrangeMultiplier}.
\begin{lem}\label{Lem:LocalMin} Let $M>0$ and consider Model (II) only. If $(\rho,D)$ satisfies~\eqref{mass_constraint} and is such that $\rho$ is given by~\eqref{rhoD} with $\phi=D-\barD$ and $\barD$ is determined by~\eqref{Eqn:LagrangeMultiplier}, then
\[
\Lambda=2\infimum{\int_\Omega v\,\rho\,(1-\rho)\,\dd x=0}{\int_\Omega v^2\,\dd x=1}\mathrm L_D[u,v]\;.
\]
As a consequence, if $(\rho,D)$ is a local minimizer of $\L$ under the mass constraint~\eqref{mass_constraint}, then $\Lambda$ is nonnegative.\end{lem}
\begin{proof} We notice that $\mathrm L_D[u,v]=\int_\Omega v\,(\LinearizedEnergyOp\,v)\,\dd x$ holds true as soon as $u=v\,\rho\,(1-\rho)$ with $\rho$ given by~\eqref{rhoD}. In particular, this is the case if $(\rho,D)$ is a local minimizer of $\L$. With $v$ fixed, an optimization of $\mathrm L_D[u,v]$ with respect to $u$ shows that $u=v\,\rho\,(1-\rho)$. When $(\rho,D)$ is a local minimizer of $\L$, it is straightforward to check that $\mathrm L_D[u,v]$ cannot be negative.\end{proof}

\subsection{Entropy--entropy production}

Along the linearized flow~\eqref{LinearizedFlow}, we have
\be{Eqn:Entropy-EntropyProduction}
\frac d{dt}\mathrm L_D[u(t,\cdot),v(t,\cdot)]=-\,2\,\mathrm I_D[u(t,\cdot),v(t,\cdot)]
\ee
where
\[
\mathrm I_D[u,v]:=\frac 12\int_\Omega\rho\,(1-\rho)\,\Big|\nabla\Big(\frac u{\rho\,(1-\rho)}-v\Big)\Big|^2\,\dd x+\frac 12\int_\Omega\big|-\,\kappa\,\Delta v+\delta\,v-u\big|^2\,\dd x\;.
\]
Let us define the bilinear form
\[
\scalar{(u_1,v_1)}{(u_2,v_2)}=\int_\Omega\(\frac{u_1\,u_2}{\rho\,(1-\rho)}-(u_1\,v_2+u_2\,v_1)\)\dd x+\kappa\int_\Omega\nabla v_1\cdot\nabla v_2\,\dd x+\delta\int_\Omega v_1\,v_2\,\dd x\;,
\]
which is such that
\[
2\,\mathrm L_D[u,v]=\scalar{(u,v)}{(u,v)}\;.
\]
\begin{lem}\label{Lem:SelfAdjointness} Consider Model (II) only and assume that $(\rho,D)$ is a local minimizer of $\L$ under the mass constraint~\eqref{mass_constraint}. On the orthogonal of the kernel of $\LinearizedEnergyOp$ with $\phi=D-\barD$, $\barD$ satisfying \eqref{Eqn:LagrangeMultiplier}, $\scalar\cdot\cdot$ is a scalar product and $\LinearOp$ is a self-adjoint operator with respect to $\scalar\cdot\cdot$. Moreover, if $(u,v)$ is a solution of \eqref{LinearizedFlow}, then
\[
\frac d{dt}\mathrm L_D[u,v]=\,-\,2\,\mathrm I_D[u,v]=\scalar{(u,v)}{\LinearOp\,(u,v)}\le0\;.
\]
As a consequence, on the orthogonal of the kernel of $\LinearizedEnergyOp$, $(0,0)$ is the unique stationary solution of \eqref{LinearizedFlow} and any solution with initial datum in the orthogonal of the kernel converges to $(0,0)$.
\end{lem}
With a slight abuse of notations, we have denoted by the \emph{kernel of $\LinearizedEnergyOp$} the set $\{(u,v)\,:\,v\in\mathrm{Ker}(\LinearizedEnergyOp)\}$.

\begin{proof} The positivity of $\mathrm I_D$ is a consequence of the definition and self-adjointness results from the computation
\begin{multline*}
-\,\scalar{(u_1,v_1)}{\LinearOp\,(u_2,v_2)}=\int_\Omega\rho\,(1-\rho)\,\nabla\Big(\frac{u_1}{\rho\,(1-\rho)}-v_1\Big)\cdot\nabla\Big(\frac{u_2}{\rho\,(1-\rho)}-v_2\Big)\,\dd x\\
+\int_\Omega\big(-\,\kappa\,\Delta v_1+\delta\,v_1-u_1\big)\,\big(-\,\kappa\,\Delta v_2+\delta\,v_2-u_2\big)\,\dd x\;.\hspace*{12pt}
\end{multline*}\vspace*{-12pt}
\end{proof}

Note that one has to take special care of the kernel of $\mathrm \LinearOp$. If $(\rho_M, D_M)$ is a stationary solution of \eqref{rho}--\eqref{D} depending diffentiably on the mass parameter $M$, it is always possible to diffentiate $\rho_M$ and $D_M$ with respect to $M$ and get a non trivial element in the kernel of $\mathrm \LinearOp$. However, it is not guaranteed that this element generates the kernel of $\LinearizedEnergyOp$ and although not observed numerically, it cannot be excluded that secondary bifurcations occur on branches of plateau-like solutions.

If $(\rho,D)$ is a stationary solution of \eqref{rho}--\eqref{bcD}, we can of course still consider $\mathrm I_D[u,v]$ and its sign determines whether $(\rho,D)$ is dynamically stable or not. In this paper we are interested in the evolution according to the nonlinear flow given by~\eqref{rho}--\eqref{bcD}. The fundamental property of mass conservation \eqref{mass_constraint} can still be observed at the level of the linearized equations~\eqref{LinearizedFlow}. The reader is invited to check that any classical solution of~\eqref{LinearizedFlow} is indeed such that
\[
\frac d{dt}\int_\Omega u(t,x)\;\dd x=0
\]
and it makes sense to impose that $\int_\Omega u\,\dd x=0$ at $t=0$. If we linearize the problem at a stationary solution given by~\eqref{rhoD}, it also makes sense to consider the constraint $\int_\Omega v\,\rho\,(1-\rho)\,\dd x=0$.

\subsection{Dynamic criterion}

After these preliminary observations, we can define two notions of stability. We shall say that a critical point $\phi$ of $\E$ is \emph{variationally stable} (resp. unstable) if and only if $\Lambda>0$ (resp. $\Lambda<0$) where $\Lambda$ is defined by~\eqref{Eqn:lambda}. Alternatively, we shall say that a stationary solution $(\rho,D)$ of \eqref{rho}--\eqref{bcD} is \emph{dynamically stable} (resp. unstable) if and only if
\[
\infimum{\int_\Omega u\,\dd x=0}{\int_\Omega v^2\,\dd x=1}\mathrm L_D[u,v]
\]
is positive (resp. negative) in case of Model (II). The operator $\LinearOp$ being self-adjoint, dynamical stability means \emph{variational stability} of $\L$ on the product space, once mass constraints are taken into account. Most of the remainder of this section is devoted to this issue.

For Model (I) we can extend the notion of dynamical stability (resp. instability) by requesting that $\inf\big\{\mathrm{Re}(\lambda)\,:\,\lambda\in\mathrm{Spectrum}(\LinearOp)\big\}$ is positive (resp. negative). However, in the case of Model (I), notions of dynamical and variational instability are not so well related, as we shall see in Section~\ref{Sec:Numerics}.

Let us start by the following observation. In the case of Models~(I) and~(II), the kernel of the operator~$\LinearizedEnergyOp$ associated to the linearized energy functional and defined by \eqref{Ephirho} determines a subspace of the kernel of~$\LinearOp$.
\begin{lem}\label{Lem:CriticalPoints} Let $\barD\in\R$ and assume that $\phi$ is a critical point of $\E$. Then $\rho$ given by \eqref{rhoD} and $D=\phi+\barD$ provides a stationary solution of \eqref{rho}--\eqref{bcD}. If $v$ is in the kernel of $\LinearizedEnergyOp$, then $(u,v)$ is in the kernel of~$\LinearOp$ if $u=\rho\,(1-\rho)\,v$.\end{lem}
\begin{proof} Using \eqref{Ephirho}, it is straightforward to check that $0=\LinearizedEnergyOp\,v=\LinearOp^{(2)}\,(u,v)$ if $v\in \mathrm{Ker}(\LinearizedEnergyOp)$. Then
\[
\LinearOp^{(1)}\,(u,v)=\nabla\cdot\left[\rho\,(1-\rho)\,\nabla\Big(\frac u{\rho\,(1-\rho)}-v\Big)\right]=0
\]
because of the special choice $u=\rho\,(1-\rho)\,v$.\end{proof}

Since \eqref{rho} preserves the mass, it makes sense to impose $\int_\Omega u\,\dd x=0$. This also suggest to consider the constraint $\int_\Omega\rho\,(1-\rho)\,v\,\dd x=0$, which has already been taken into account in \eqref{Eqn:lambda}. Let us give some more precise statements, in the case of Model (II). First we can state a more precise version of Lemma~\ref{Lem:LocalMin}. Let us define
\[
\Lambda_1:=\,2\,\infimum{\int_\Omega u\,\dd x=0}{\int_\Omega v^2\,\dd x=1}\mathrm L_D[u,v]\;.
\]
\begin{lem}\label{Lem:ACharacterization} Let $M>0$. Consider Model (II) only and assume that $(\rho,D)$ is a critical point of $\L$ under the mass constraint~\eqref{mass_constraint}. With $\phi=D-\barD$ where $\barD$ is the unique real number determined by \eqref{Eqn:LagrangeMultiplier}, consider~$\Lambda$ defined by \eqref{Eqn:lambda}. Then we have $\Lambda_1\le\Lambda$. If either $\Lambda<\delta$ or $\Lambda_1<\delta$, then we have $\Lambda=\Lambda_1$.\end{lem}
\begin{proof} If $(\rho,D)$ is a critical point of $\L$, the analysis of Section~\ref{Sec:parametrization} shows that $\rho$ is given by~\eqref{rhoD} with $\phi=D-\barD$ and $\barD$ determined by~\eqref{Eqn:LagrangeMultiplier}. Consider first the minimization problem
\[
\infimum{\int_\Omega v\,\rho\,(1-\rho)\,\dd x=0}{\int_\Omega v^2\,\dd x=1}\mathrm L_D[u,v]\;.
\]
As in Lemma~\ref{Lem:LocalMin}, the optimization with respect to $u$ shows that $u=v\,\rho\,(1-\rho)$ and it is then straightforward to get that $2\,\mathrm L_D[u,v]=\int_\Omega v\,(\LinearizedEnergyOp\,v)\,\dd x=\Lambda$. Additionally, we know that $v$ solves the Euler-Lagrange equation
\be{Opt}
\LinearizedEnergyOp\,v=-\,\kappa\,\Delta v+\delta\,v-v\,\rho\,(1-\rho)=\Lambda\,v-\mu\,\rho\,(1-\rho)
\ee
for some Lagrange multiplier $\mu$ and we have $\int_\Omega u\,\dd x=\int_\Omega v\,\rho\,(1-\rho)\,\dd x=0$. This proves that $\Lambda_1\le\Lambda$.

Now, consider a minimizer $(u,v)$ for $\Lambda_1$. We find that
\[
u=(v-\bar v)\,\rho\,(1-\rho)\quad\mbox{with}\quad\bar v:=\frac{\int_\Omega v\,\rho\,(1-\rho)\,\dd x}{\int_\Omega\rho\,(1-\rho)\,\dd x}\;.
\]
Moreover $v$ solves the Euler-Lagrange equation
\be{Opt2}
-\,\kappa\,\Delta v+\delta\,v-v\,\rho\,(1-\rho)=\Lambda_1\,v-\bar v\,\rho\,(1-\rho)\;.
\ee

Hence we have found that $2\,\mathrm L_D[u,v]=\kappa\int_{\Omega}|\nabla v|^2\,\dd x+\delta\int_{\Omega}v^2\,\dd x-\int_{\Omega}\rho\,(1-\rho)\,|v-\bar v|^2\,\dd x$, so that
\begin{multline*}
\Lambda_1-\delta=\inf_{v\not\equiv0}\frac{\kappa\int_{\Omega}|\nabla v|^2\,\dd x-\int_{\Omega}\rho\,(1-\rho)\,|v-\bar v|^2\,\dd x}{\int_{\Omega}v^2\,\dd x}\\=\infimum{\int_\Omega v\,\rho\,(1-\rho)\,\dd x=0}{v\not\equiv0\,,\;c\in\R}\frac{\kappa\int_{\Omega}|\nabla v|^2\,\dd x-\int_{\Omega}\rho\,(1-\rho)\,v^2\,\dd x}{\int_{\Omega}v^2\,\dd x+c^2}\\
=\infimum{\int_\Omega v\,\rho\,(1-\rho)\,\dd x=0}{v\not\equiv0}\frac{\kappa\int_{\Omega}|\nabla v|^2\,\dd x-\int_{\Omega}\rho\,(1-\rho)\,v^2\,\dd x}{\int_{\Omega}v^2\,\dd x}\;,
\end{multline*}
where the last equality holds under the condition that either $\Lambda<\delta$ or $\Lambda_1<\delta$. Hence we have shown that $\Lambda_1-\delta=\Lambda-\delta$, which concludes the proof.\end{proof}

\begin{rem} With no constraint, it is straightforward to check that $\delta$ is an eigenvalue of $\LinearOp$, and $(u,v)=(0,1)$ an eigenfunction. Hence, as soon as $\Lambda_1<\delta$, we have that $\int_\Omega u\,\dd x=0$ if $(u,v)$ is a minimizer for~$\Lambda_1$, because of \eqref{Opt2}. This justifies why the condition of either $\Lambda<\delta$ or $\Lambda_1<\delta$ enters in the statement of Lemma~\ref{Lem:ACharacterization}. \end{rem}

In the case of Model~(II), we can get a bound on the growth of the unstable mode.
\begin{cor}\label{Cor:DynamicalVsVariational} Consider Model (II) only and assume that $(\rho,D)$ is a critical point of $\L$ under the mass constraint~\eqref{mass_constraint}. If $\Lambda$ defined by \eqref{Eqn:lambda} is negative, then we have
\[
\infimum{\int_\Omega u\,\dd x=0}{\int_\Omega v^2\,\dd x=1}\frac{\mathrm I_D[u,v]}{\mathrm L_D[u,v]}\le\Lambda
\]
and the growth rate of the most unstable mode of~\eqref{LinearizedFlow} is at least $2\,|\Lambda|$.
\end{cor}
\begin{proof} Consider a function $v$ given by \eqref{Opt} with $\int_\Omega v\,\rho\,(1-\rho)\,\dd x=0$, $\int_\Omega v^2\,\dd x=1$, $u=v\,\rho\,(1-\rho)$ and take $(u,v)$ as a test function. Then
\[
\frac{\mathrm I_D[u,v]}{\mathrm L_D[u,v]}=\frac{\int_\Omega(\LinearizedEnergyOp\,v)^2\,\dd x}{\int_\Omega v\,(\LinearizedEnergyOp\,v)\,\dd x}=\frac{\int_\Omega(\Lambda\,v-\mu\,\rho\,(1-\rho))^2\,\dd x}{\int_\Omega(\Lambda\,v-\mu\,\rho\,(1-\rho))\,v\,\dd x}=\Lambda+\frac{\mu^2}\Lambda\int_\Omega\rho^2\,(1-\rho)^2\,\dd x\le\Lambda\;.
\]
Using \eqref{Eqn:Entropy-EntropyProduction}, if $\mathrm L_D[u,v]$ is negative, then we get $\frac d{dt}\mathrm L_D[u,v]\le\,-\,2\,\Lambda\,\mathrm L_D[u,v]$, thus proving that $\mathrm L_D[u,v](t)\le\mathrm L_D[u,v](0)\,e^{2\,|\Lambda|\,t}$ for any $t\ge0$.\end{proof}

The result of Corollary~\ref{Cor:DynamicalVsVariational} on the most unstable mode can be rephrased in terms of standard norms. By definition of $\mathrm L_D$, we get that $\int_\Omega (u^2+v^2)\,\dd x\ge2\int_\Omega u\,v\,\dd x\ge2\,|\mathrm L_D[u,v]|\ge2\,|\mathrm L_D[u,v](0)|\,e^{2\,|\Lambda|\,t}$ for any~$t\ge0$.

Summarizing, we have shown the following result.
\begin{thm}\label{Thm:Main2} Let $M>0$ and consider the case of Model (II). Assume that $(\rho,D)$ is a stationary solution of \eqref{rho}--\eqref{bcD} such that~\eqref{mass_constraint} is satisfied and let $\phi=D-\barD$ with $\barD$ satisfying \eqref{Eqn:LagrangeMultiplier}. Then the following properties hold true.
\begin{enumerate}
\item[(i)] Neither dynamical instability nor variational instability can occur if $(\rho,D)$ is a local minimizer of~$\L$ under the mass constraint \eqref{mass_constraint} or, equivalently, if $\phi$ is a local minimizer of $\E$ such that \eqref{mass_constraint} and~\eqref{rhoD} hold.
\item[(ii)] If $(\rho,D)$ is a local minimizer of $\L$ under the mass constraint \eqref{mass_constraint}, then any solution $(u,v)$ of \eqref{LinearizedFlow} converges towards $(0,0)$ when the initial datum is assumed to be in the orthogonal of the kernel of~$\LinearOp$ and with sufficiently low energy.
\item[(iii)] Dynamical stability implies variational stability.
\item[(iv)] Variational instability and dynamical instability are equivalent and, with above notations, $\Lambda_1=\Lambda$.
\end{enumerate}
\end{thm}
On the contrary, no clear relation between variational and dynamical (in)stability is known in the case of Model (I), except the result of Lemma~\ref{Lem:CriticalPoints}, which is not so easy to use from a numerical point of view.

\section{Numerical results}\label{Sec:Numerics}

Let us summarize our findings on radial stationary solutions of \eqref{rho}--\eqref{bcD}, with parameters $\delta$ and $\kappa$ in the range discussed in Section~\ref{Sec:NumericalRange}, when $\Omega$ is the unit ball in $\R^d$, with $d=1$ or $d=2$. Our results deal with either Model (I) or Model (II), defined respectively by \eqref{ModelI} and \eqref{ModelII}, as follows:
\begin{enumerate}
\item[(i)] We compute the branches of monotone, non-constant, radial solutions that bifurcate from constant solutions for the two models, in dimensions $d=1$ and $d=2$.
\item[(ii)] We study variational and dynamical stability of these solutions. The two notions coincide for Model (II), which is partially explained with the help of the Lyapunov functional.
\item[(iii)] Dynamical stability holds up to the turning point of the branch when it is parametrized by the mass for Model (II) in dimensions $d=1$ and $d=2$. This is also true in dimension $d=1$ for Model~(I).
\item[(iv)] In dimension $d=1$, the variational stability of the branch of monotone, non-constant solutions is more restrictive than the dynamical stability in case of Model (I).
\end{enumerate}
Before entering in the details, let us observe that bifurcation diagrams are more complicated in dimension $d=2$ than for $d=1$, and that the lack of a Lyapunov functional makes the study of Model (I) significantly more difficult.

\medskip All computations are based on the shooting method presented in Proposition~\ref{Prop:Parametrization}. This allows us to find all radially symmetric stationary solutions, as the range of parameter $a$ for which solutions exist is bounded according to Corollary~\ref{cor:bounds}. Hence we are left with a single ordinary differential equation, which can be solved using standard numerical methods. Because of the smallness of the parameter $\kappa$, the shooting criterion $\varphi_a'(1)=0$ has a rather stiff dependence on $a$. This makes directly finding all zeros of the criterion for a given $\phi_0$ difficult, so in practice we use perturbation and continuation methods to parametrize the whole branch of monotone, plateau-like solutions.

The computation of the spectrum of the linearized evolution operator \eqref{LinearizedFlow} is done using a basis of cosines, normalized and scaled to meet the boundary conditions. This allows for fast decomposition of the coefficients by FFT. In the case $d=2$, such a basis in not orthogonal, which is taken into account using a mass matrix during diagonalization. In cases where the constraints cannot be enforced directly at the basis level, a Rayleigh quotient minimization step is done, on the orthogonal of the constrained space.

Numerical computations have been made entirely using the \texttt{NumPy} and \texttt{SciPy} Python libraries, freely available from \texttt{http://scipy.org}. These make use of reference numerical libraries \texttt{LAPACK} and \texttt{odepack}.

We start by considering constant solutions and make use of the notations of Section~\ref{Sec:Constants}.

\begin{figure}[ht]\includegraphics[width=.49\textwidth]{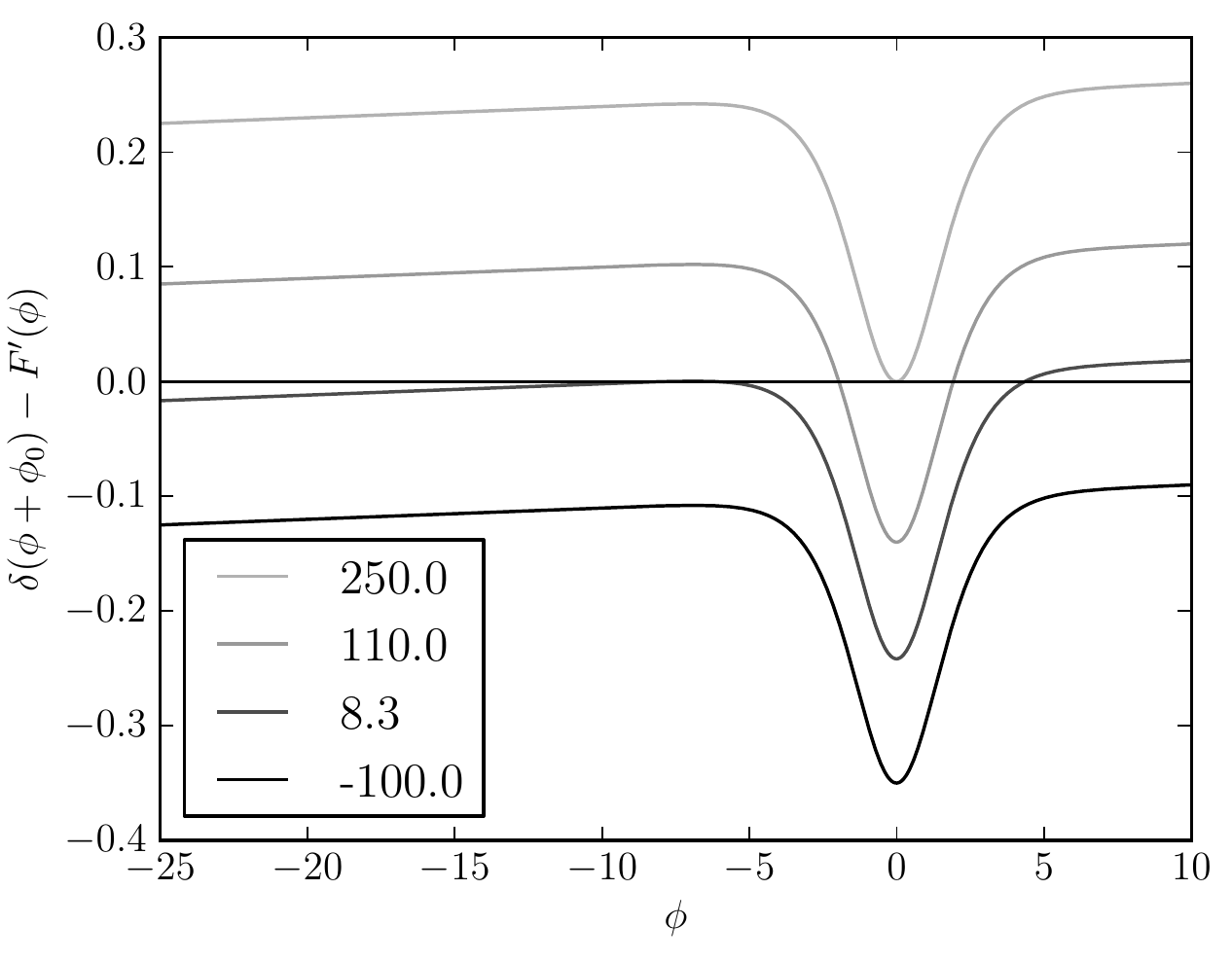}\caption{\small Plot of $\phi\mapsto\delta\,(\phi+\barD)-f(\phi)=\delta\,(\phi+\barD)-F'(\phi)=\delta\,(\barD-k(\phi))$ for various values of $\barD$. Each zero of the function provides a constant stationary solution of \eqref{rho}--\eqref{bcD}. The plot shown here corresponds to Model (I), with $\delta=10^{-3}$.}\label{Fig:CstSoln}\end{figure}
\begin{figure}[ht]\includegraphics[width=.49\textwidth]{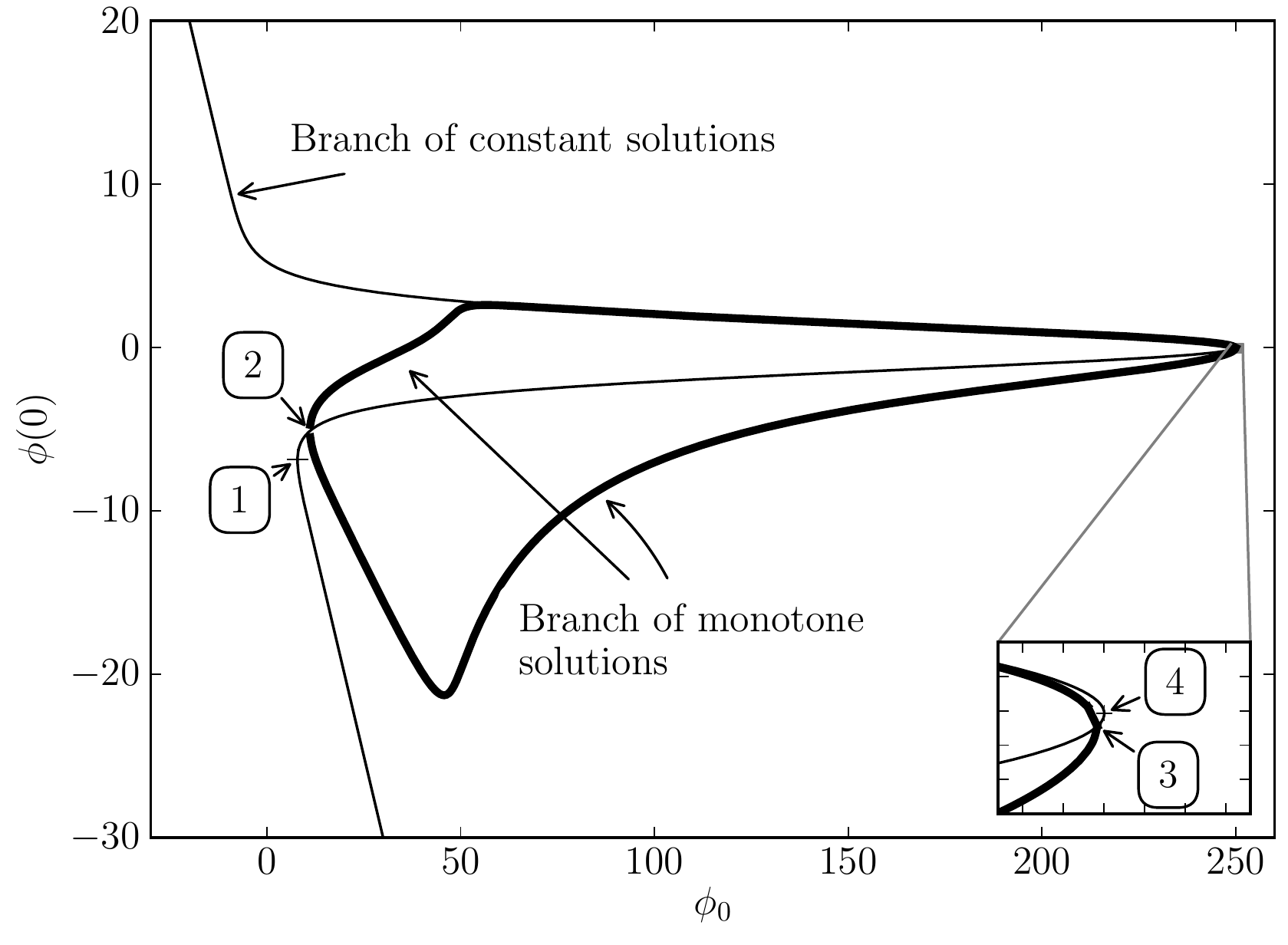}\caption{\small Parametrization by $\barD$ of the branches of solutions in case of Model (I), $d=1$, with $\delta=10^{-3}$, $\kappa=5\times10^{-4}$ and $\Omega=(0,1)$. There are either one or three constant solutions for a given value of $\phi_0$. Strictly monotone solutions correspond to the bold curve. Notice that on the upper part of the graph the two branches are close but distinct.}\label{Fig:Branches}\end{figure}
Let us comment on the plots of Fig.~\ref{Fig:Branches}.
\begin{itemize}
\item[] \ovalbox{1} First turning point: $\barD=\phi_0^-$, on the branch of constant solutions: for lower values of $\barD$, there is only one constant solution $\phi\equiv\phi(0)$, which converges to $+\infty$ as $\barD\to-\infty$.
\item[] \ovalbox{2} and \ovalbox{3} Non constant solutions bifurcate from constant solutions, which are unstable in the corresponding interval for $\barD$. The solutions of the two branches correspond to monotone solutions, either increasing or decreasing, and always bounded from above and from below by constant solutions.
\item[] \ovalbox{4} Second turning point: $\barD=\phi_0^+$, on the branch of constant solutions: for higher values of $\barD$, there is only one constant solution $\phi\equiv\phi(0)$, which converges to $-\infty$ as $\barD\to+\infty$.
\end{itemize}

The dependence of plateau-like solutions on parameters $\barD$ and $\kappa$ is shown in Fig.~\ref{phi_plots}.
\vspace*{-6pt}
\DeuxFigures{ht}{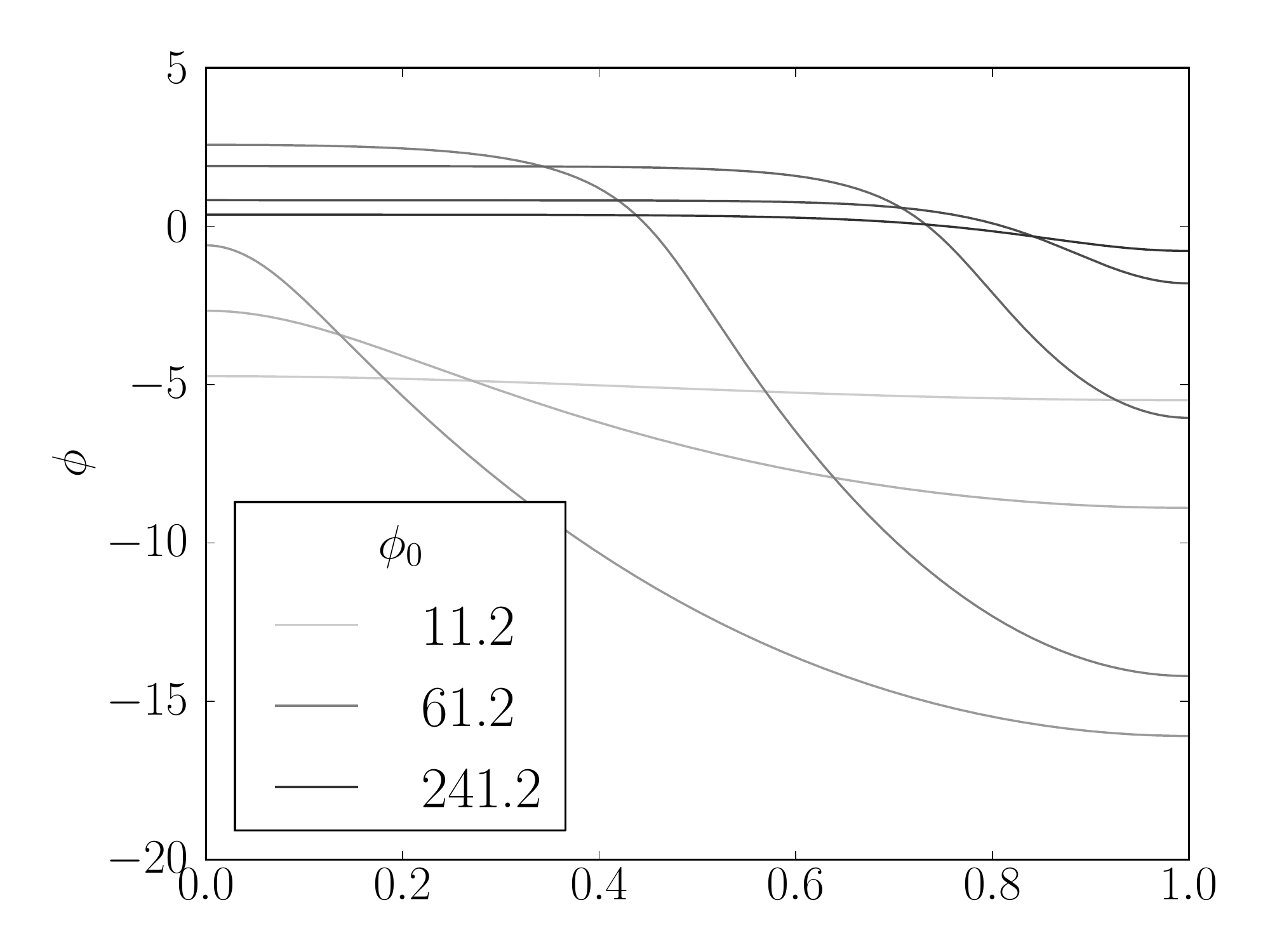}{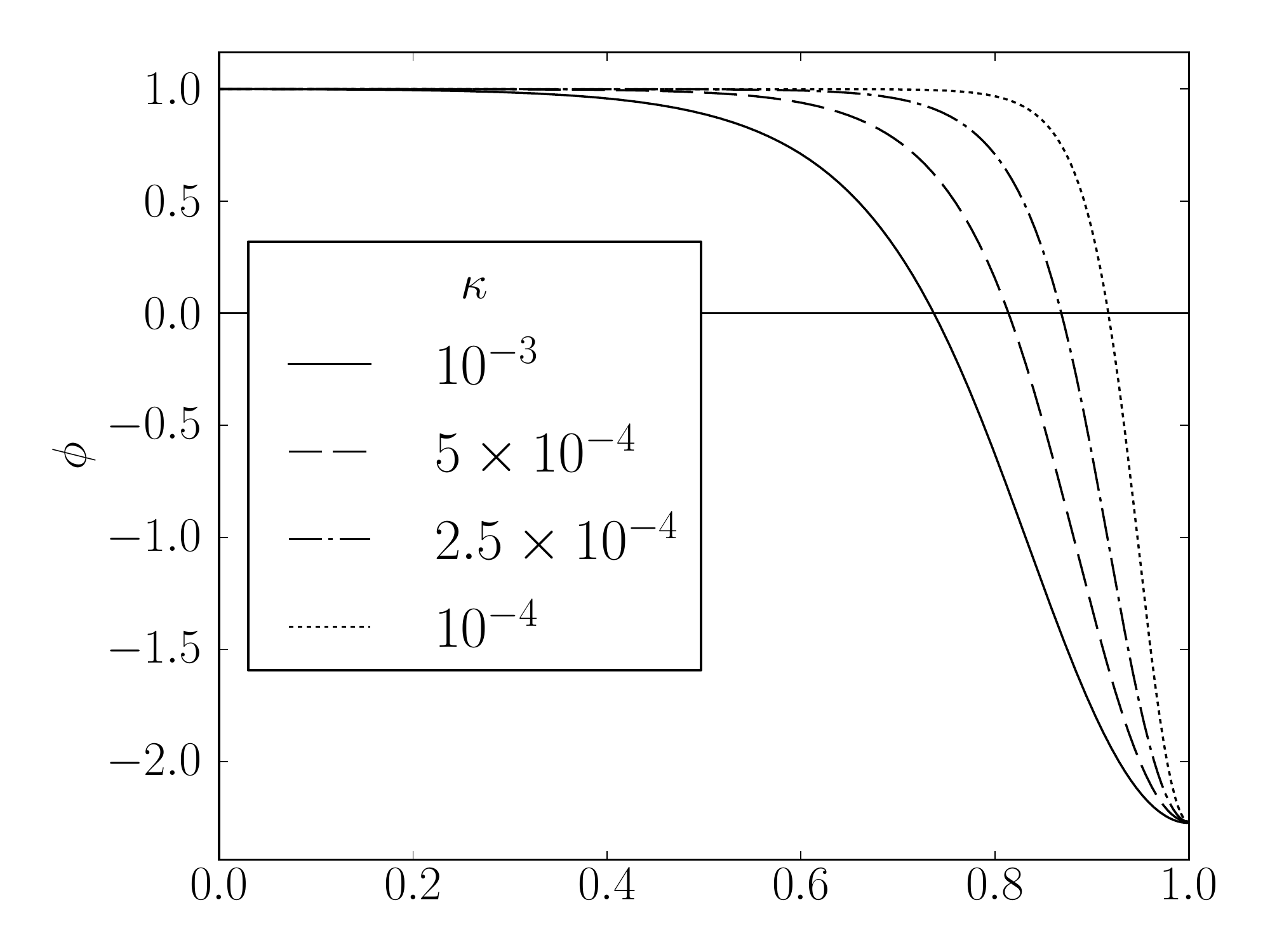}{In the case of Model (I), $d=1$, $\delta=10^{-3}$, we consider various profiles for $x\mapsto\phi(x)$ with $x\in(0,1)=\Omega$ either (left) as $\barD$ varies and $\kappa=5\times10^{-4}$, or (right) as $\kappa$ varies, with $\phi(0)=1$.}{phi_plots}

Next we consider monotone, plateau-like solutions. In Figs.~\ref{Fig:5-6} and~\ref{Fig7-8}, the shaded region corresponds to masses for which constant solutions are unstable.

\DeuxFigures{ht}{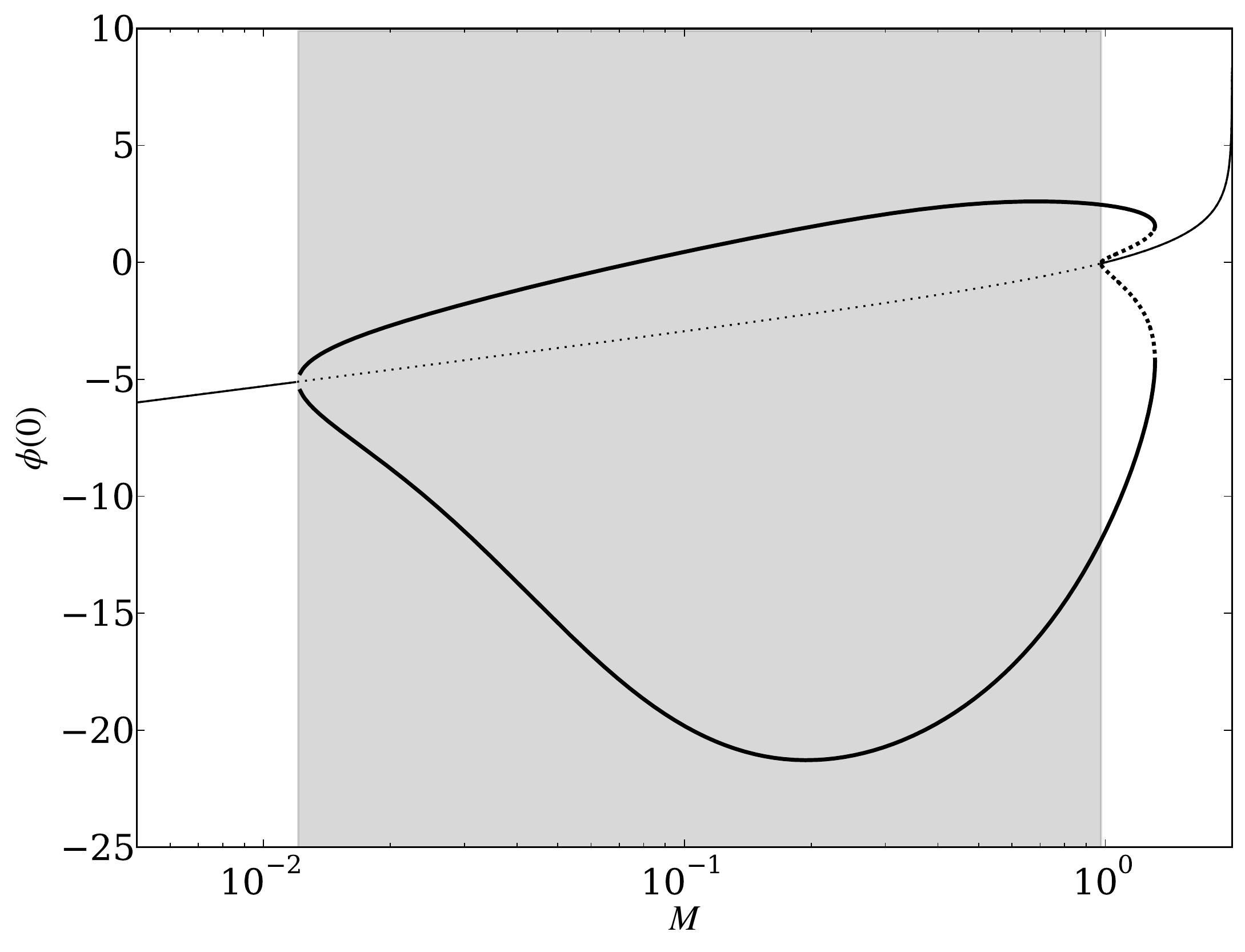}{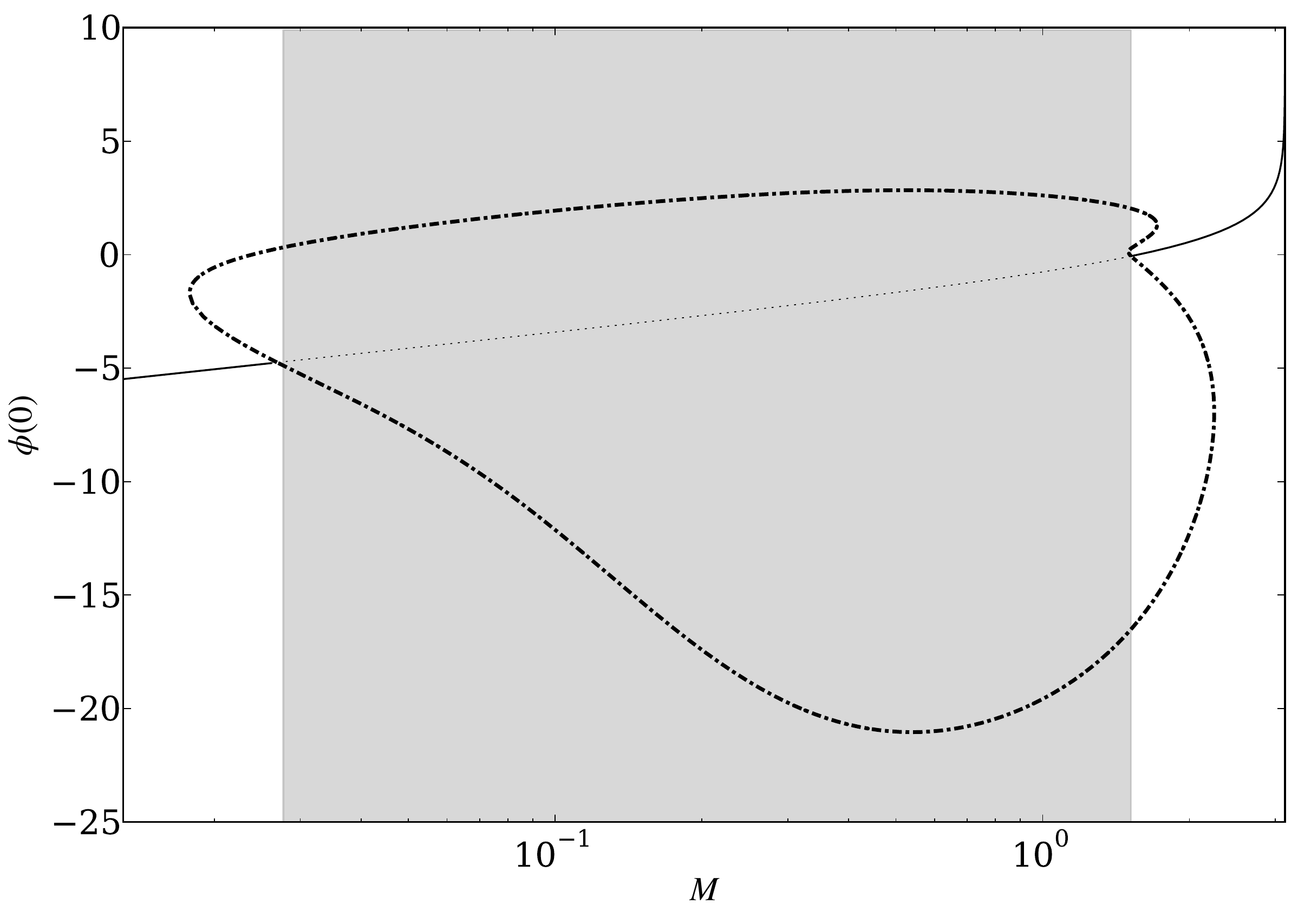}{Model (I), $\kappa=5\times10^{-4}$, $\delta=10^{-3}$. Thin lines represent constant solutions and bold ones the plateau-like solutions. For readability purposes we use a logarithmic scale for the mass. \emph{Left:} $d=1$. The dotted part of each branch shows where solutions are dynamically unstable. \emph{Right:} $d=2$.}{Fig:5-6}
\DeuxFigures{ht}{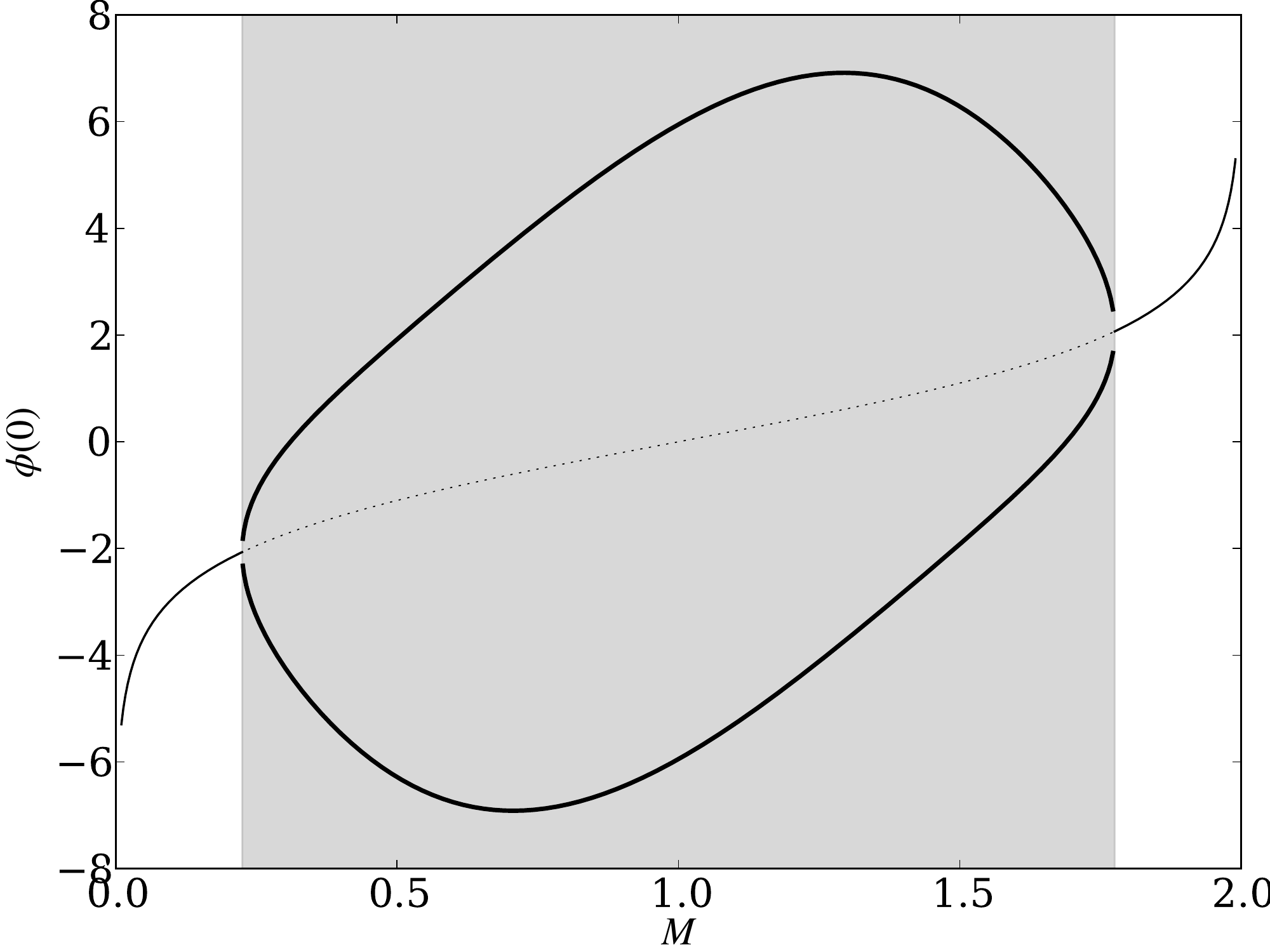}{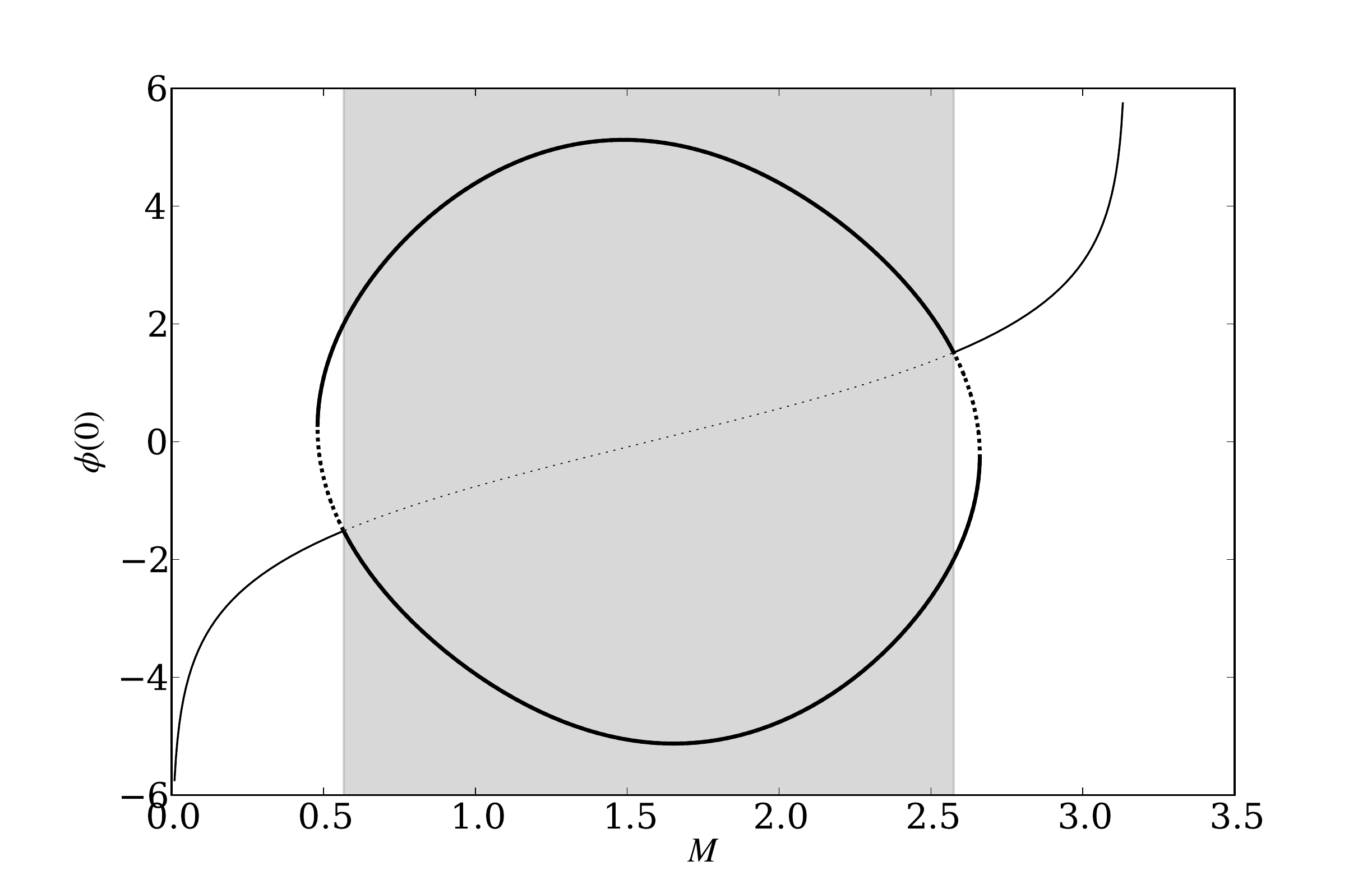}{Model (II), $\kappa=10^{-2}$, $\delta=10^{-3}$: thin lines represent constant solutions and bold ones the plateau-like solutions. The dotted part of each branch shows where solutions are dynamically unstable. \emph{Left:} $d=1$. \emph{Right:} $d=2$.}{Fig7-8}

Dynamical and variational stability criteria and their interplay are a tricky issue, especially in case of Model (I) in which we have no theoretical framework to relate the two notions. See Fig.~\ref{Fig:10}.
\vspace*{-6pt}

\DeuxFigures{ht}{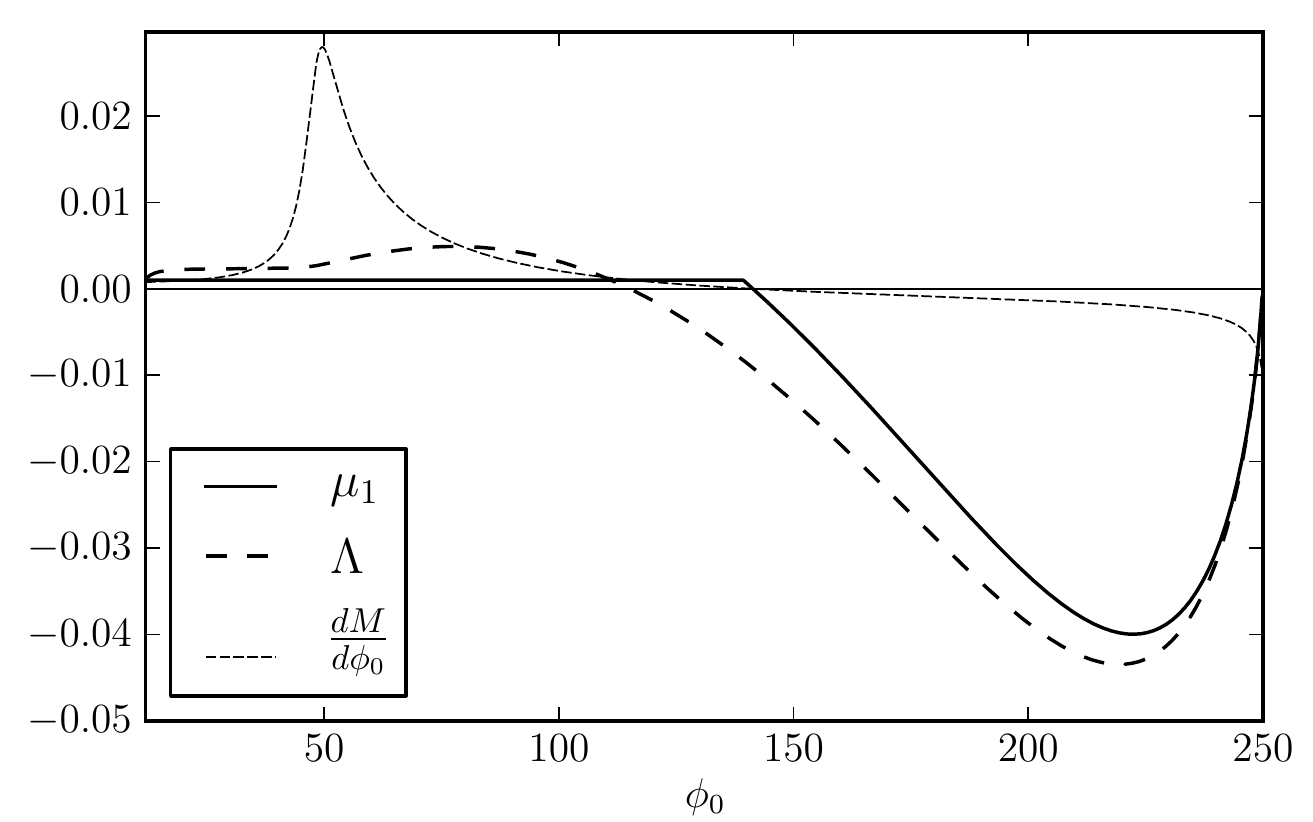}{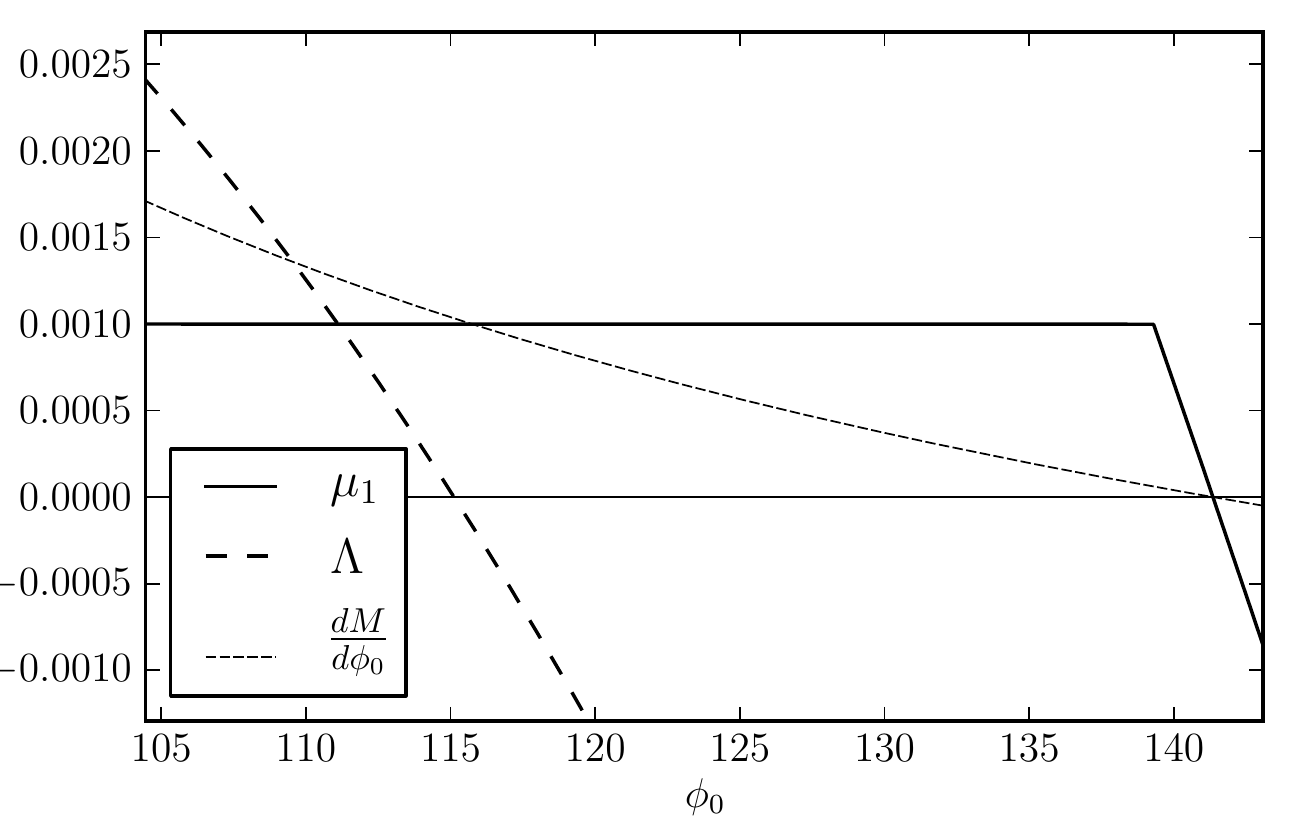}{Model (I), $d=1$. We numerically compare the criteria for variational and dynamical instability along the branch of monotone, non-constant solutions. When $dM/d\barD$ changes sign, this means that the branch has a turning point when plotted in terms of $M$. We observe that this turning point corresponds to the loss of dynamical stability, while variational stability is lost for smaller values of $\barD$ along the branch: see in particular the enlargement (right). Here $\mu_1$ corresponds to the lowest value of $\mathrm{Re}(\langle(u,v),-\LinearOp\,(u,v)\rangle)$ under the constraints \hbox{$\langle(u,v),(u,v)\rangle=1$} and $\int_\Omega u\,\dd x=0$, and $\langle\cdot,\cdot\rangle$ denotes the standard scalar product.}{Fig:10}

Stationary solutions are critical points of $\E$. It is therefore interesting to determine whether they are minima or not, either for fixed values of $\barD$ or for fixed values of $M$, which makes more sense from the dynamical point of view. However, only in the case of Model (II) minimizers of $\E$ are also minimizers of~$\L$ and therefore dynamically and variationally stable. See Figs.~\ref{Fig:13-14} and~\ref{Fig:15-16}.

\TroisFigures{ht}{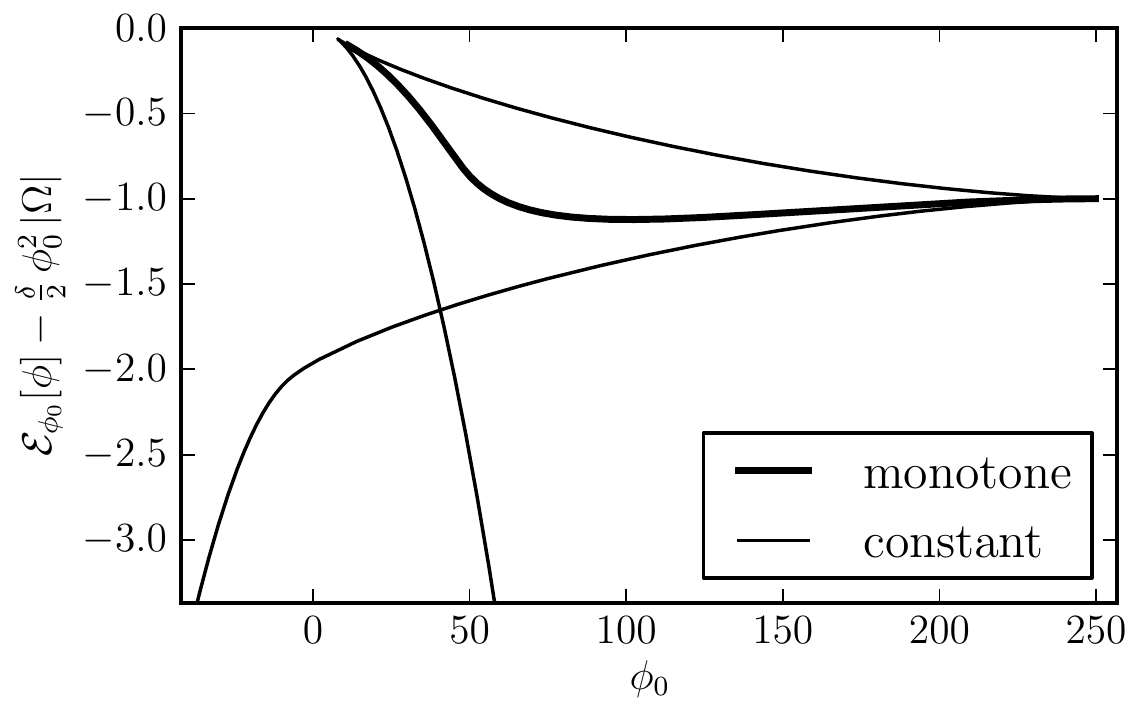}{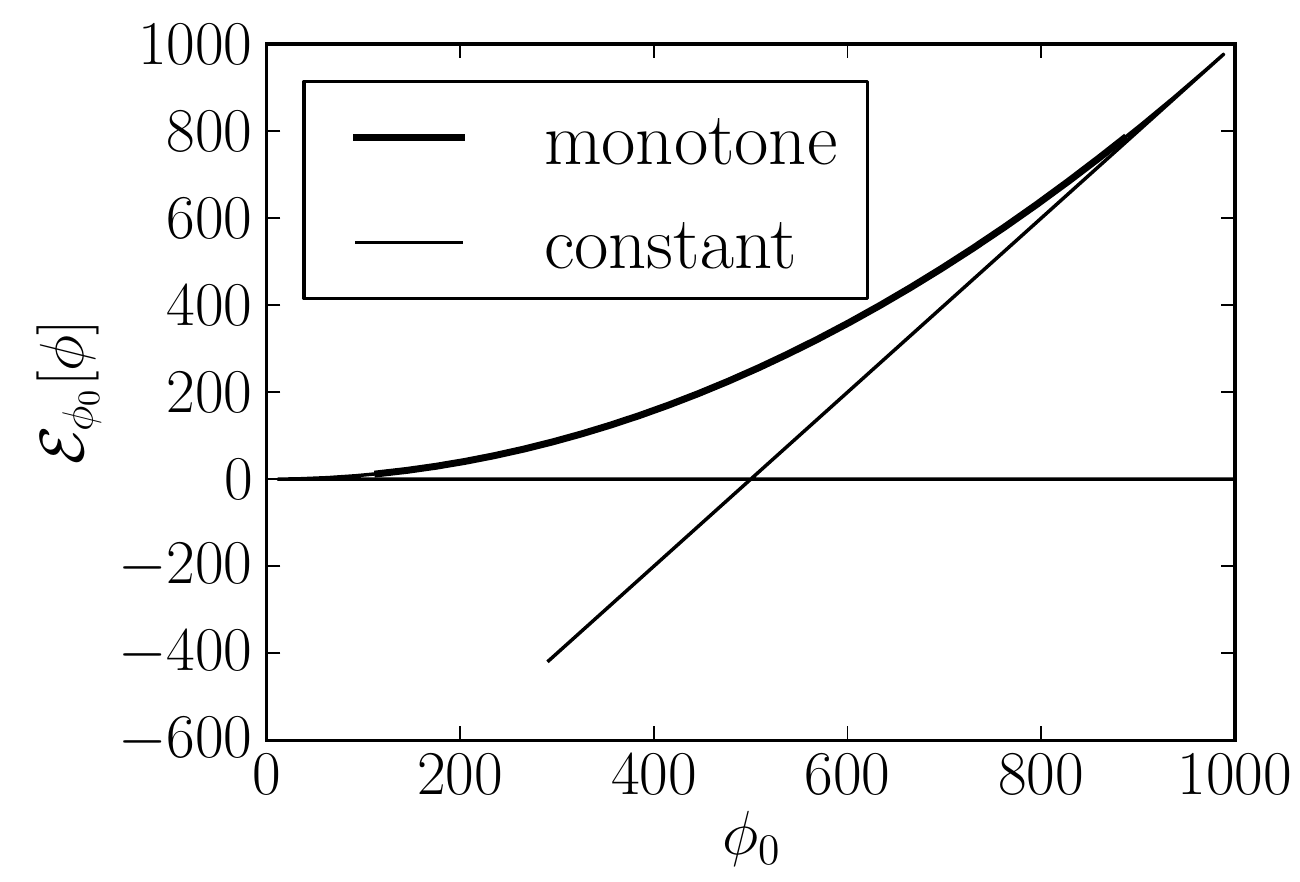}{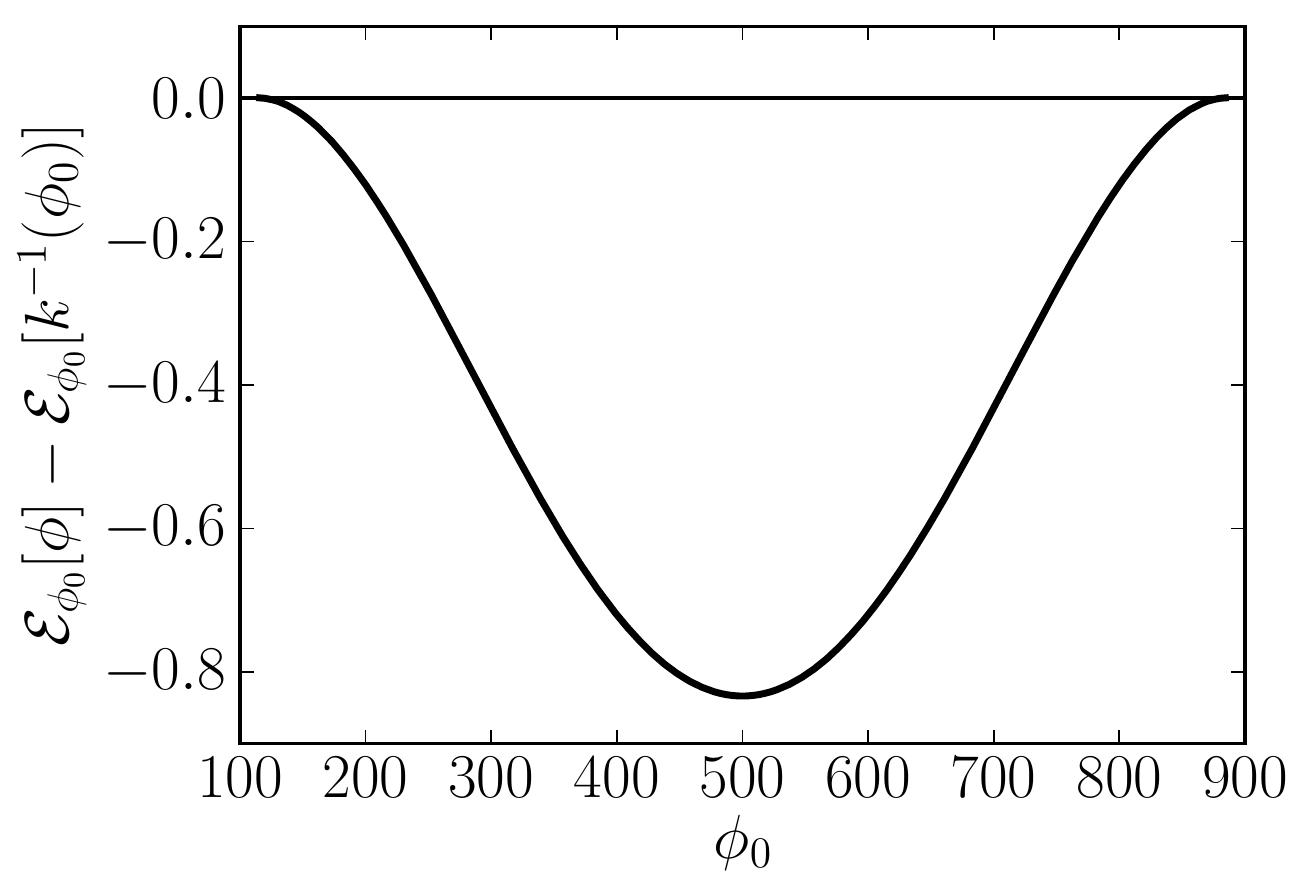}{The energy is represented as a function of $\barD$ for constant and monotone (either increasing or decreasing) solutions. Here we assume $d=1$. \emph{Left:} Model (I), the energy $\E$ is shifted by $\frac\delta 2\,\phi_0^2\,|\Omega|$. \emph{Center:} Model (II). Non constant solutions (upper curve) are undistinguishable from a branch of constant solutions. \emph{Right:} Details for Model (II): difference of the energies of the constant and non-constant solutions (under appropriate restrictions on $\barD$).}{Fig:13-14}

\DeuxFigures{ht}{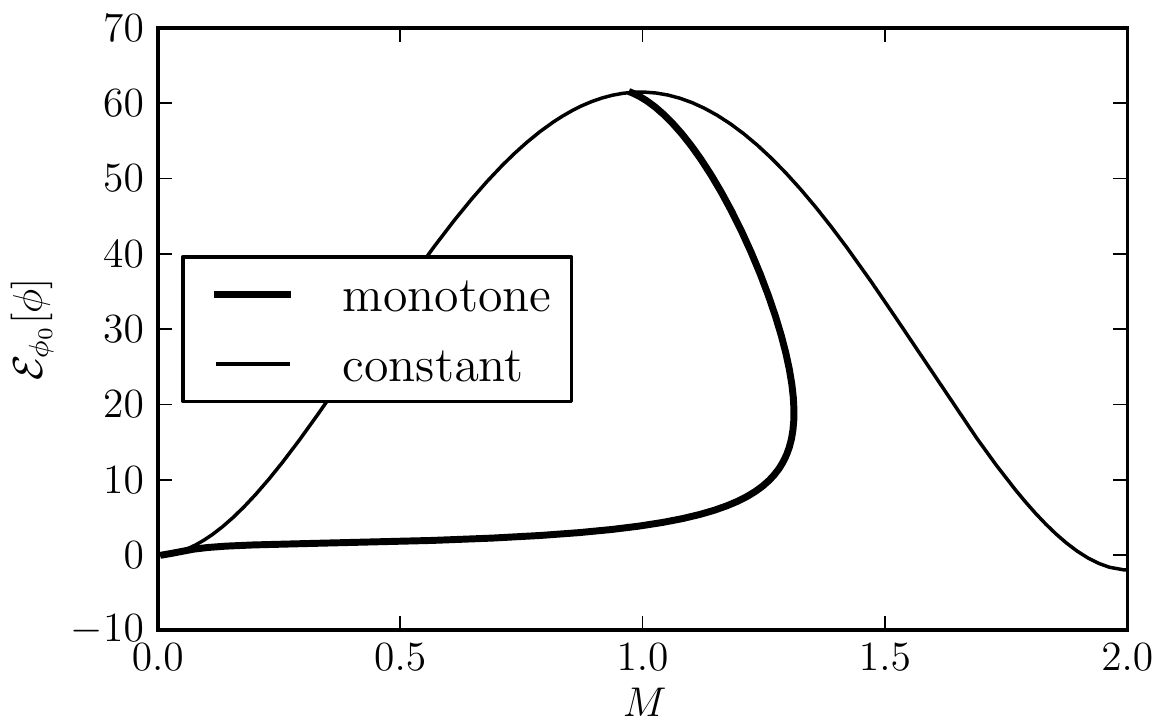}{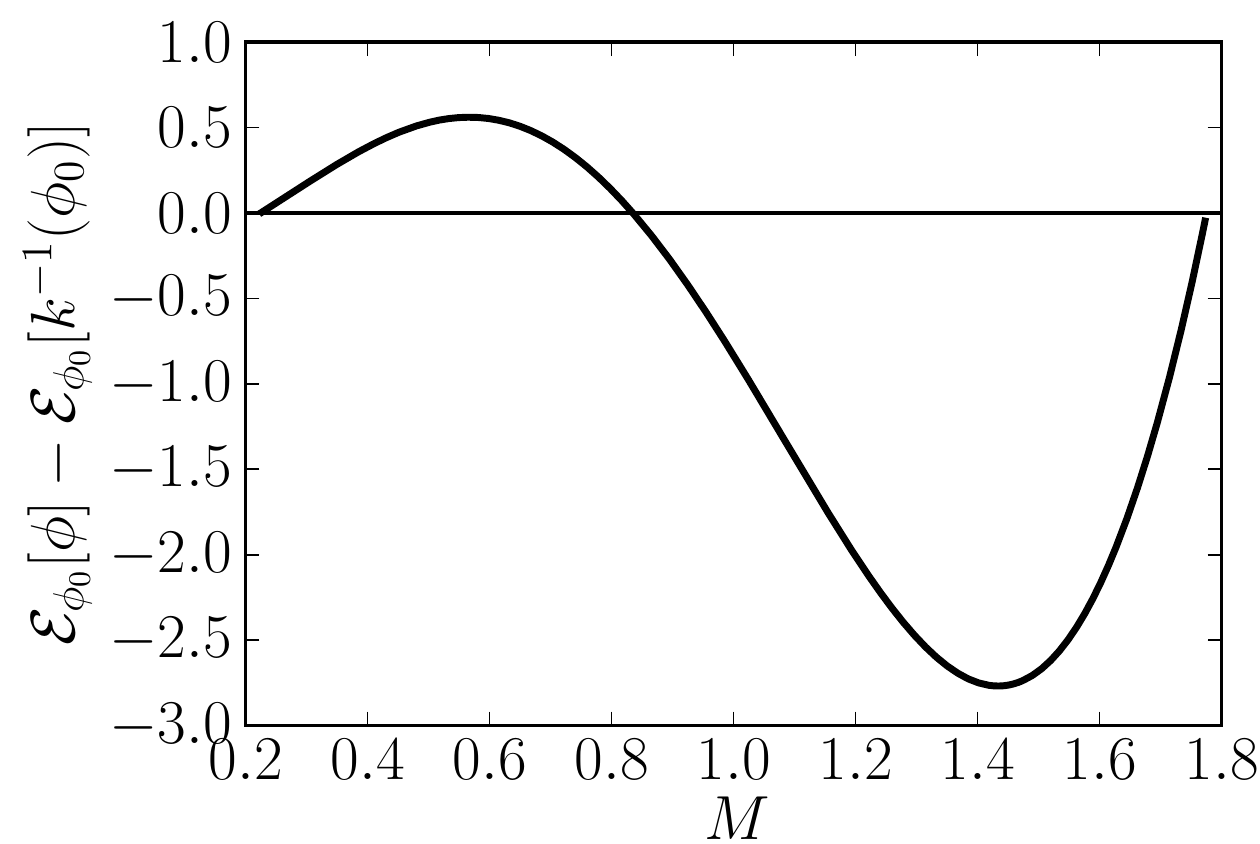}{For any given mass, there is exactly one constant solution. Hence minimizers of $\F[D]=\E[\phi+\barD]$ with $\barD=\barDM[D]$ for masses $M$ in a certain range are not constant. \emph{Left:} Model (I), $d=1$. \emph{Right:} Model (II), $d=1$. These minimizers are also minimizers of the Lyapunov functional and therefore dynamically stable ($\barD$ is restricted to an appropriate range).}{Fig:15-16}

Finally in case of Model (II), we can check that dynamical and variational stability are compatible, see Fig.~\ref{Fig:17-19}.
\vspace*{-6pt}
\DeuxFigures{ht}{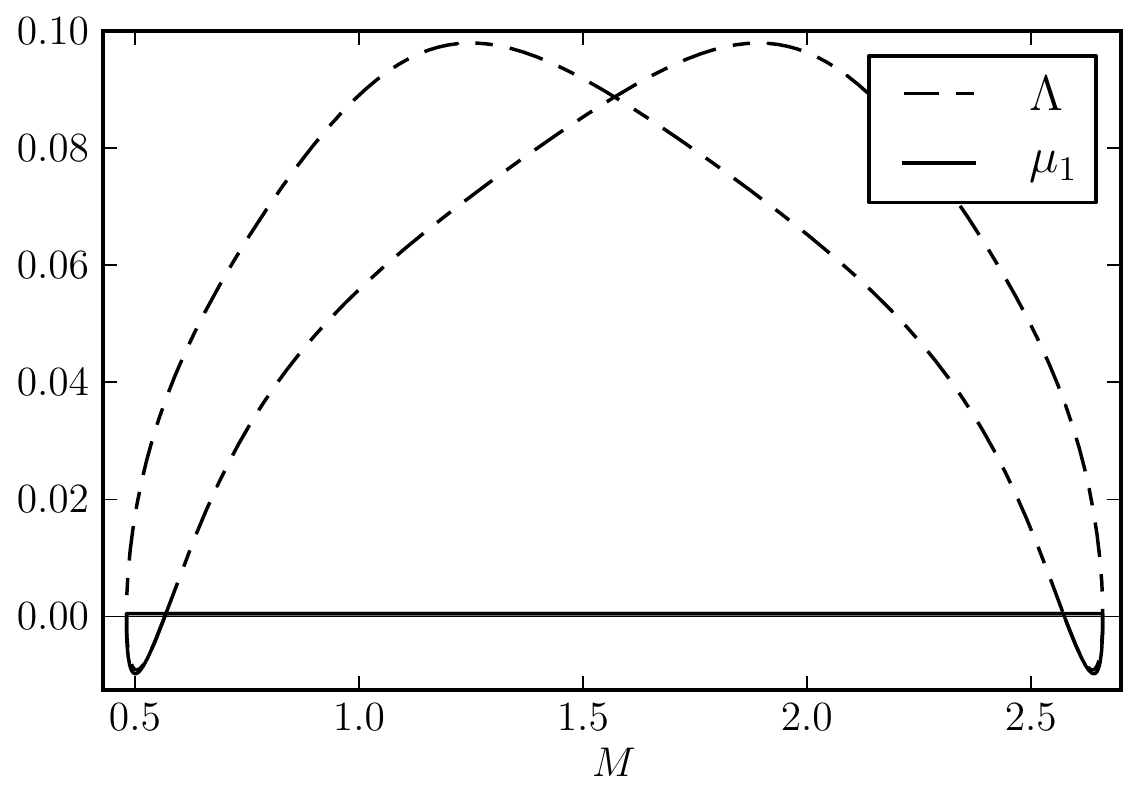}{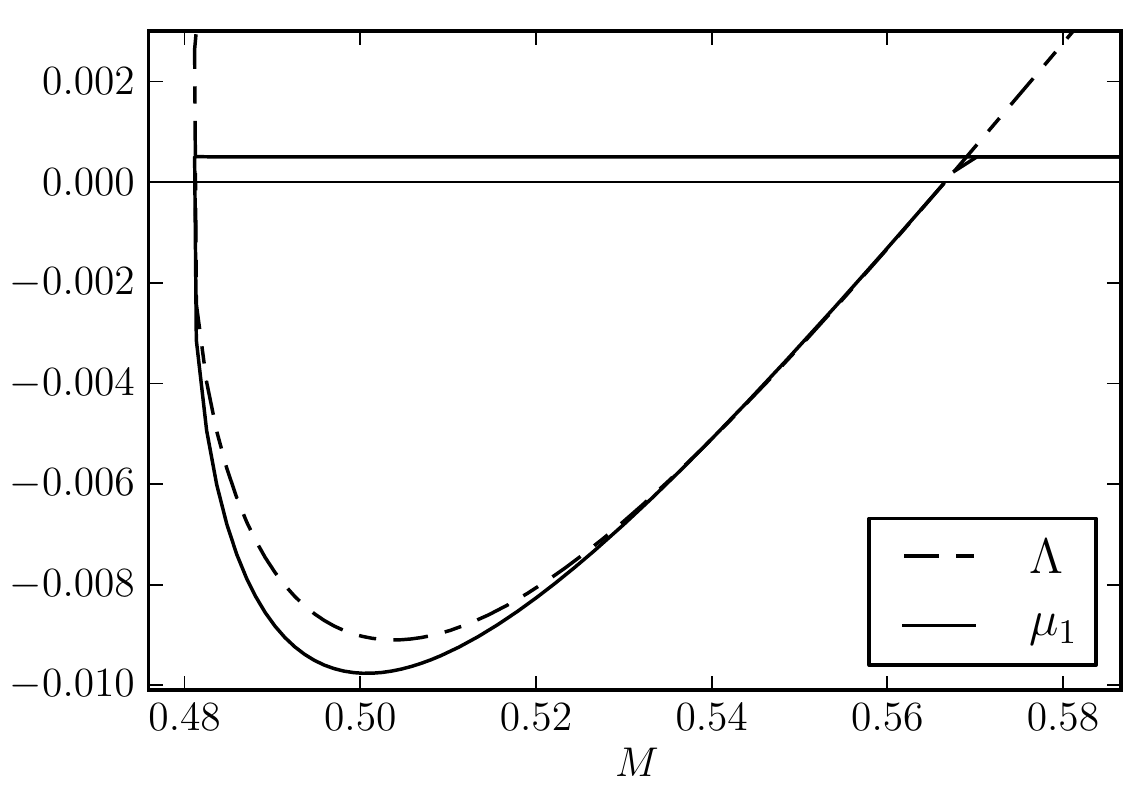}{Model (II), $d=2$. \emph{Left:} Solution of $M\mapsto\Lambda$, where, for each $M$, we compute the two monotone plateau-like solutions, and then $\Lambda$ according to \eqref{Eqn:lambda}. Hence $\Lambda<0$ means that the solution is variationally unstable under the mass constraint. \emph{Right:} Detail is shown. Here $\mu_1$ corresponds to the lowest value of $\langle(u,v),-\LinearOp\,(u,v)\rangle$ under the constraints $\langle(u,v),(u,v)\rangle=1$ and $\int_\Omega u\,\dd x=0$.}{Fig:17-19}

\section{Concluding remarks}\label{Sec:Conclusion}

{\bf Model (II)} is the (formal) gradient flow of the Lyapunov functional $\L$ with respect to a distance corresponding to Wasserstein's distance for $\rho$ and an $L^2$ distance for~$D$ (see \cite{Blanchet-Laurencot,Carrillo-Lisini,laurencot2011gradient} for further considerations in this direction). Critical points of $\L$ are stationary solutions for the system, they attract all solutions of the evolution equation and the infimum of $\L$ is achieved by a monotone function, which is therefore either a plateau solution or a constant solution. When $d=1$ numerics, at least for the values of the parameters we have considered, show that plateau solutions exist only in the range in which constant solutions are unstable and are uniquely defined in terms of the mass, but when $d=2$, the range for dynamically stable plateau solutions is larger than the range (in terms of the mass) of constant unstable solutions under radial perturbation. Infima of $\L$ and $\E$ actually coincide. Consistently with our analysis, we find that the linearized evolution operator around minimizing solutions has only positive eigenvalues. Moreover, this operator is self-adjoint in the norm corresponding to the quadratic form given by the second variation of $\L$ around a minimizer. Hence, when $d=2$, we observe the existence of multiple stable (under radial perturbations) stationary solutions.

In case of {\bf Model (I)}, no Lyapunov functional is available, to our knowledge. Still, all stationary states are characterized as critical points of $\E$ and obtained (as long as they are radially symmetric) using our shooting method. In dimension $d=1$, the structure of the set of solutions is not as simple as in Model~(II), and this can be explained by the frustration due to the $\rho\,(1-\rho)$ term in the equation for $D$. Numerically, when $d=1$, we observe that monotone plateau solutions are uniquely defined and dynamically stable in the range where constant solutions are dynamically unstable. However, when $d=1$, we also have a range in which both types of solutions are dynamically stable, which means that the system has no global attractor. We do not even know whether stationary solutions attract all solutions of the evolution problem or not.

To give a simple picture of the physics involved in the two models of crowd modeling studied in this paper, we may use the following image. The potential $D$ defines the \emph{strategy} of the individuals. It takes into account the source term (the density $\rho$ in case of Model~(II) and $\rho\,(1-\rho)$ in case of Model~(I)) to determine a preferred direction. Because it is governed by a parabolic equation, it takes the value of the source term into account not only at instant $t$, but also in the past, which means that there is a memory effect. Of course, recent past receives a larger weight, and actually two mechanisms are at work to update the system: a local damping, with time scale determined by $\delta$ and a diffusion term (position of the source term gets lost on the long time range), with a time scale governed by $\kappa$. Both coefficients being small, the time scale (that is, the memory of the system) is long compared to the time scale for $\rho$.

As far as $\rho$ is concerned, the diffusion accounts for random effects while the drift is tempered by some \emph{tactical} term, which tries to avoid densely populated areas, and is taken into account by the mean of the $(1-\rho)$ term in the drift.

In case of Model (II) the strategy defined by the source term is simple: individuals want to aggregate in high $\rho$ densities. In case of Model (I) the strategy is different, as the system tends to favor regions with intermediate densities, typically $\rho$ of the order of $1/2$. Of course, this is antagonist with the trend to concentrate in regions where $D$ is large and introduces some frustration in the system. At a very qualitative level, this is an explanation for the fact that multiplicity of dynamically stable stationary state occurs in Model (I) even when $d=1$.

\par\medskip\centerline{\rule{2cm}{0.2mm}}\medskip\noindent{\it Acknowledgements.} {\small Authors have been supported by the ANR project CBDif-Fr. J.D.~and P.M.~thank King Abdullah University of Science and Technology (KAUST) for support. The authors thank the two referees who have suggested significant improvements.}

\medskip\noindent{\scriptsize\copyright~2013 by the authors. This paper may be reproduced, in its entirety, for non-commercial~purposes.}


\end{document}